                    \def\version{April 7, 2021}                       %
 \documentclass[reqno,11pt]{amsart}
 \usepackage{amsmath, amsthm, a4, latexsym, amssymb}
\usepackage[unicode]{hyperref}
\usepackage{color}

 \usepackage{mathrsfs}

\def\@rmrk#1#2{\refstepcounter
    {#1}\@ifnextchar[{\@yrmrk{#1}{#2}}{\@xrmrk{#1}{#2}}}

\makeatletter\@addtoreset{equation}{section}\makeatother

 \newfont{\bfit}{cmbxti10 scaled 1200}
 
 \newtheorem{theorem}{Theorem}[section]
\newtheorem{lemma}[theorem]{Lemma}
\newtheorem{proposition}[theorem]{Proposition}
\newtheorem{cor}[theorem]{Corollary}
\newtheorem{rem}[theorem]{Remark}

\newcommand{\N}{{\mathbb{N}}}

\newcommand{\R}{{\mathbb{R}}}
\newcommand{{\rd}}{\R^d}
\newcommand{\IP}{{\mathbb P}}
\newcommand{\E}{\mathbb E}

\newcommand{\8}{\infty}

\newcommand{\eu }{{\bf e_1}}

\renewcommand{\b}{\beta}
\renewcommand{\L}{\mathscr L^{\ssup\sigma}_T(x)}

\newcommand{\rmd}{\mathrm{d}}
\newcommand{\D}{\Delta}

\newcommand{\e}{\varepsilon}

\newcommand{\tht}{\theta}

\newcommand{\dd}{\,\text{\rm d}}             % a straight d for differentials
\newcommand{\dB}{\xi}
\newcommand{\vphi}{\varphi}
\newcommand{\nn}{\nonumber}

\newcommand{\fC }{{\mathfrak C}}

\newcommand{\sZ}{{\mathscr Z}}

\newcommand{\heap}[2]{\genfrac{}{}{0pt}{}{#1}{#2}}

\newcommand{\ssup}[1] {{\scriptscriptstyle{({#1}})}}

\newcommand{\hh}{{\mathfrak h_\e^{\mathrm{st}}}}

\newcommand{\HH}{{\mathscr H^{\mathrm{st}}}}

\makeatletter
\def\section{\@startsection{section}{1}{\z@}{-3.5ex plus -1ex minus 
 -.2ex}{2.3ex plus .2ex}{\bf}}
\def\subsection{\@startsection{subsection}{2}{\z@}{-3.25ex plus -1ex minus 
 -.2ex}{1.5ex plus .2ex}{\bf}}
\makeatother

\newcommand{\cvlaw}{\stackrel{\rm{ law}}{\longrightarrow}}
\newcommand{\eqlaw}{\stackrel{\rm{ law}}{=}}
\newcommand{\cvfidi}{\stackrel{\rm{ f.d.m.}}{\longrightarrow}}

\def\thebibliography#1{\section*{References}
  \list%
  {\arabic{enumi}.}%                          {\star}{\star}{\star} style of reference number {\star}{\star}{\star}
    {\settowidth\labelwidth{[#1]}\leftmargin\labelwidth
    \advance\leftmargin\labelsep
    \parsep0pt\itemsep0pt
    \usecounter{enumi}}
    \def\newblock{\hskip .11em plus .33em minus .07em}
    \sloppy                   % \clubpenalty4000\widowpenalty4000
    \sfcode`\.=1000\relax}

\begin{document}
%%%%%%%%%%%%%%%%%%%%%%%%%%%%%%%%%%%%%%%%%%%%%%%
\title[Fluctuation of KPZ equation in $d\geq 3$ and the Gaussian free field]
{\large Space-time fluctuation of the Kardar-Parisi-Zhang equation in $d\geq 3$ and the Gaussian free field}
\author[Francis Comets, Cl\'ement Cosco and Chiranjib Mukherjee]{}
\maketitle
\thispagestyle{empty}
\vspace{-0.5cm}

\centerline{\sc  Francis Comets \footnote{Universit\'e de Paris,
Laboratoire de Probabilit\'es, Statistique et Mod\'elisation,
 LPSM (UMR 8001 CNRS, SU, UP)
B\^atiment Sophie Germain, 8 place Aur\'elie Nemours, 75013 Paris, {\tt comets@lpsm.paris}}, 
Cl\'ement Cosco\footnote{Department of Mathematics, Weizmann Institute of Science, Herzl St 234, Rehovot, Israel, {\tt  clement.cosco@weizmann.ac.il}} and Chiranjib Mukherjee\footnote{University of M\"unster, Einsteinstrasse 62, M\"unster 48149, Germany, {\tt chiranjib.mukherjee@uni-muenster.de}}}
\renewcommand{\thefootnote}{}

\vspace{-0.5cm}
\centerline{\textit{Universit\'e de Paris, Weizmann Institute of Science and University of M\"unster}}
\vspace{0.2cm}

\begin{center}
\version
\end{center}
\smallskip

\centerline{\it Dedicated to the memory of Dima Ioffe}
\smallskip

\begin{quote}{\small {\bf Abstract: }
We study the solution $h_\e$ of the Kardar-Parisi-Zhang (KPZ) equation in $\rd \!\times\! [0, \infty)$ for $d \geq 3$:
$$
\frac{\partial}{\partial t} h_{\e} = \frac12 \D h_{\e} +  \bigg[\frac12   |\nabla h_\e |^2  - C_\e\bigg]+
 \b \e^{\frac{d-2}2}   \dB_{\e} \;,\qquad\,\,   h_{\e}(0,x) =0.
$$
Here $\b>0$ is a parameter called the disorder strength, $\xi_\e=\xi\star \phi_\e$ is a spatially smoothened (at scale $\e$) Gaussian space-time white noise 
and $C_\e$ is a divergent constant as $\e\to 0$. When
$\b$ is sufficiently small and $\e\to 0$, $h_\e(t,x)- \hh(t,x)\to 0$ in probability
where $\hh(t,x)$ is the {\it stationary solution} of the KPZ equation -- more precisely, $\hh(t,x)$
 solves the above equation with a random initial condition (that is independent of the driving noise $\xi$) and its law is constant in $(\e,t,x)$. 
In the present article we quantify the rate of the above convergence in this regime and show that the fluctuations 
$(\e^{1-\frac d2} [h_\e(t,x) - \hh(t,x)])_{x\in \rd, t > 0}$ {\it about} the stationary solution  
converges pointwise (with finite dimensional distributions in space and time)  to a Gaussian free field evolved by the deterministic heat equation. We also identify the fluctuations {\it of} 
the stationary solution itself and show that the rescaled averages $\int_{\rd} \dd x \, \varphi(x)\,\,\e^{1-\frac d2} [\hh(t,x) - \E\hh(t,x)]$ converge 
to that of the {\it stationary solution} of the stochastic heat equation with additive noise, but with (random) {\it Gaussian free field marginals} (instead of flat initial condition).}
%$$\mathscr H(t,x)=\gamma(\beta) \int_0^\infty \int_{\rd} \rho(\sigma+t,y-x) \, \xi(\sigma,z) \dd \sigma \, \dd z$$
%with $\rho(\sigma,x)$ being the standard heat kernel. The limiting process $\mathscr H$ is also the (real-valued) solution of 
%the {\it{non-noisy heat equation}} $\partial_t \mathscr H=\frac 12 \Delta \mathscr H$ with a {\it random initial condition} $\mathscr H(0,x)$ given by a {\it{Gaussian free field}} on $\rd$.
%We further obtain convergence of the spatial averages for test functions $ \varphi\in \mathcal C^\infty_c(\rd)$:
%$$\int_{\rd} \dd x \, \varphi(x)\,\,\e^{1-\frac d2} [h_\e(t,x) - \mathfrak h_\e(t,x)] \cvlaw \int \dd x \,  \varphi(x)) \mathscr H(t,x).$$ }
\end{quote}

%%  F added Dec.6
\setcounter{footnote}{0}
\renewcommand{\thefootnote}{\alph{footnote}}

%%%%%%%%\newpage

\section{Introduction and background.}\label{sec:background}
We consider the {\it{Kardar-Parisi-Zhang}} (KPZ) equation 
written informally as
 \begin{equation} \label{eq:KPZ}
 \frac{\partial}{\partial t} h = \frac12 \D h +  \bigg[ \frac12   |\nabla h |^2 -\infty\bigg] +  \beta \dB \qquad u\colon \rd\times \mathbb R_+\to \mathbb R, 
 \end{equation}
and driven by a totally uncorrelated Gaussian space-time white noise $\dB$ with $\beta>0$ being a constant  known as the {\it disorder}. The above equation enjoys a huge popularity as {\it{the default model}} of 
stochastic growth in $(d+1)$-dimensions \cite{KPZ86,T18}. For  $d=1$ this equation also becomes relevant for non-equilibrium fluctuations and appears as the scaling limit of front propagation of exclusion processes and weakly asymmetric interacting particles \cite{BG97,SS10,C12, DGP17} as well as that of the free energy of the discrete directed polymer \cite{AKQ14,CSZ16} at intermediate disorder. The KPZ equation \eqref{eq:KPZ} in $d=1$ has now seen a huge upsurge of interest in the recent years and a vast amount of deep mathematical results are now available, see  
\cite{SS10,ACQ11,Q12} 
and \cite{H13,H14,GIP15,GP17,K17} for a rigorous construction of the notion of a solution. 
For spatial dimension $d>1$, due to the irregular nature of the noise $\xi$, a precise construction of a solution to the SPDE \eqref{eq:KPZ} is not expected. On the other hand, despite being ill-posed for larger dimensions, the KPZ equation still remains relevant for random surface growth and has its own appeal even in the so-called {\it small disorder regime} -- a distinguishing feature of this equation in higher dimensions, see \cite{MSV18} for recent work 
in the physics literature.

In the present context we fix a spatial dimension $d\geq 3$ and let $\xi$ denote space-time white noise -- that is, with $\mathcal S(\R_+\times \rd)$ denoting the space of smooth and rapidly decreasing functions at infinity, $\{\dB(\vphi)\}_{\varphi\in \mathcal S(\R_+ \times \rd)}$ is a family of Gaussian random variables $\dB(\vphi)= \int_0^\8 \int_{\rd} \dd t \, \dd x \,\, \dB(t,x)\,\vphi (t,x)$ with mean $0$ and covariance $\mathbb E[ \dB(\varphi_1)\,\, \dB(\varphi_2)]= \int_0^\8 \int_{\rd} \varphi_1(t,x) \varphi_2(t,x) \dd t \dd x$. 
Consider a spatially mollified version of \eqref{eq:KPZ} with flat initial condition 
\begin{equation}\label{eq:KPZe}
 \frac{\partial}{\partial t} h_{\e} = \frac12 \D h_{\e} +  \bigg[\frac12   |\nabla h_\e |^2  - C_\e\bigg]+
 \b \e^{\frac{d-2}2}   \dB_{\e} \;,\qquad\,\,   h_{\e}(0,x) =0,
\end{equation}
which is driven by the spatially mollified Gaussian noise 
$$
\xi_{\e}(t,x) = (\xi(t,\cdot)\star \phi_\e)(x)=  \int \phi_\e(x - y) \xi(t,y) \dd y.
$$
Here $\phi_\e(\cdot)=\e^{-d}\phi(\cdot/\e)$ is a suitable approximation of the Dirac measure $\delta_0$, while $\phi: \R^d \to \R_+$  is a fixed function which is smooth and spherically symmetric, with $\mathrm{supp}(\phi) \subset B(0,\frac12)$ and  $\int_{\rd} \phi(x)\dd x=1$. Also $C_\e=\b^2(\phi\star \phi)(0) \e^{-2}/2$ is a suitably chosen divergent (renormalization) constant (cf. \eqref{C-eps} below).  Thus, $\{\xi_\e(t,x)\}$ is a centered Gaussian field with covariance
\begin{equation}\label{V}
{{\mathbb E[ \xi_{\e}(t,x) \xi_{\e}(s,y) ]  = \delta({t-s}) \, \e^{-d} V\big((x-y)/\e\big),}} \quad\mbox{with  } V=\phi \star \phi,
\end{equation}
being a smooth function supported in $B(0,1)$.  

The solution $h_\e(t,x)$ to \eqref{eq:KPZe} is a priori non-stationary (i.e. its law fluctuates as the mollification scale $\e>0$, $t>0$ and $x\in \rd$ vary). However, as shown in \cite{CCM19}, when the disorder $\beta$ remains sufficiently small, a key object of interest is the {\it stationary solution} $\hh(t,x)$ which is constructed as follows. We first remark that, by the Cole-Hopf transform and the Feyman-Kac formula (see Remark \ref{rem:stationary} for details), the solution $h_\e$ of \eqref{eq:KPZe} is directly related via the almost sure identity \eqref{hZ} to 
the (logarithm of the) martingale 
\begin{equation}\label{eq:Z}
\mathscr Z_T(x)= \mathscr Z_T(\xi;x)= E_x  \bigg[ \exp\bigg\{\beta \,\int_0^{T} \int_{\rd} \, \phi(W_{
{ s}}-y)  \dB(s, y)\,\dd s \dd y -  \frac{\beta^2\,T} 2\,\, (\phi\star\phi)(0)\bigg\}\bigg]
\end{equation}
with $E_x$ denoting expectation with respect to the law $P_x$ of a $d$-dimensional Brownian path $W=(W_s)_{s\geq 0}$ starting at $x\in \rd$, which is independent of the noise $\dB$. Now if we extend 
$\xi$ also to negative times and set 
\begin{equation}\label{scaling}
 \dB^{\ssup{\e,t,x}}(s,y)= \e^{\frac{d+2}2}\dB\big( t-\e^2 s,\e y-x  \big),
\end{equation}
then by diffusive scaling, time-reversal and spatial translation of the noise, $\dB^{\ssup{\e,t,x}}$ 
\emph{is itself a Gaussian white noise and possesses the same law as that of} $\dB$. Thus we have an almost sure identity 
\begin{equation}\label{hZ}
h_\e(t,x)\stackrel{\mathrm{a.s.}}{=} \log\mathscr Z_{\frac t{\e^2}} \big(\xi^{\ssup{\e,t,0}}; \frac x \e\big).
\end{equation} 
Next, as shown in \cite{MSZ16} (see also \cite{C17}  for a general reference), since $\sZ_T$ is a non-negative martingale, there exists $\beta_c\in (0,\infty)$ such that $(\mathscr Z_T)_T$ is uniformly integrable for $\beta\in (0,\beta_c)$ and there is a strictly positive non-degenerate random variable $\mathscr Z_\infty(x)$ so that, a.s. as $T \to \8$, 
\begin{equation}\label{eq:dichotomy}
\mathscr Z_T(x) \to 
\begin{cases}
\mathscr Z_\infty(x) &\mbox{if}\,\, \beta\in (0,\beta_c),\\
0 & \mbox{if}\,\, \beta\in (\beta_c,\infty).
\end{cases}
\end{equation}
 Now for $\beta\in (0,\beta_c)$, if $\mathfrak u(\xi)$ is any arbitrary representative of the positive random limit 
$\sZ_\infty=\sZ_\infty(0)$, then we can define 
\begin{equation}\label{def:stationary} 
\hh(t,x) := \log\mathfrak u(\xi^{\ssup{\e,t,x}}),
\end{equation}
to be the aforesaid {\it stationary solution}, which then satisfies the following three properties:
\begin{enumerate}
\item $\hh(t,x)$ is a (global in time) solution to \eqref{eq:KPZe} with a random initial condition, that is independent of the driving noise $\xi$.
\item The law of $\hh(t,x)$ is invariant under both space and time translation, and under oscillations of $\e>0$  -- that is, the family $\{\hh(t,x)\}_{\e,t,x}$ is {\it constant in law}.  
\item The non-stationary solution $h_\e(t,x)$ of \eqref{eq:KPZe} with flat initial condition approximates its stationary version -- that is, a law of large numbers holds: For any $t>0$ and $x\in \rd$ and as $\e\downarrow 0$,
\begin{equation}\label{LLN}
h_\e(t,x)- \hh(t,x)\stackrel{\mathbb P}\longrightarrow 0 \quad\mbox{pointwise, without any spatial averaging.}
\end{equation}
\end{enumerate}
With this background, the goal of the current article is two-fold:  First, we develop a method for studying the fluctuations of $h_\e$ {\it about} its stationary version $\hh$. This is done by quantifying the rate of the convergence \eqref{LLN} and identifying the error explicitly {\it pointwise} in terms of finite dimensional distributions in space and time (cf. Theorem \ref{th:h} and also Theorem \ref{th:CVagainstTestFun} and Proposition \ref{prop:tightness}). The limiting process is a Gaussian free field (GFF) subject to the (deterministic) heat equation. Next, we identify the fluctuation {\it of} the stationary solution $\hh$ itself around its mean $\E[\hh]$ in a spatially averaged sense (cf. Theorem \ref{thm:stationary}), with the latter limit being the {\it stationary solution} of the stochastic heat equation with additive noise, but with the above {\it GFF marginals} (instead of flat initial condition). We turn to the precise statements of these assertions. 

\section{Main results.}\label{sec:3dKPZ}
Throughout the sequel, we write 
$\rho(t,x)=\frac 1 {(2\pi t)^{d/2}} \mathrm e^{-\frac{|x|^2}{2t}}$
 for the standard Gaussian kernel %Moreover, we write 
%\begin{equation}\label{eq:Heps}
%\mathscr H_\e(t,x)= \e^{1-\frac d2} \,\, [h_\e(t,x)- \hh(t,x)],
%\end{equation}
and
\begin{equation}\label{def-sigma-beta}
 \begin{aligned}
&\gamma^2(\beta)= \b^2 \int_{\rd} \dd y \,\, V(y)\,\,E_y\bigg[\mathrm e^{\beta^2\int_0^\infty V( W_{2s})\,\dd s}\bigg]
\end{aligned}
\end{equation}
which diverges for $\beta$ large (recall from \eqref{V} that $V=\phi\star \phi$). Moreover, for a sequence of time-space random fields we denote by $\cvfidi$ the convergence in the sense of finite dimensional marginal distributions in time and space. Finally, with the solution $h_\e$ to \eqref{eq:KPZe} with $h_\e(0,\cdot)=0$ and the stationary solution $\hh$ defined 
in \eqref{def:stationary} we set 
\begin{equation}\label{eq:Heps}
\mathscr H_\e(t,x)= \e^{1-\frac d2} \,\, [h_\e(t,x)- \hh(t,x)],
\end{equation}
Here is the statement of our first main result. 
\begin{theorem}[Space-time fluctuations {\it about} the stationary solution]\label{th:h}
 Assume $d\geq 3$. Then there exists $\beta_0\in (0,\b_c)$ such that for $\beta<\beta_0$ and as $ \e \to 0$
  \begin{equation}\label{eq:th1}
\begin{aligned}
 &\big(\mathscr H_\e(t,x) \big)_{x\in \rd, t>0} \,\cvfidi \, \big(\mathscr H(t,x)\big)_{x\in \rd, t>0}\\
\quad\mbox{where}
&\quad\mathscr H(t,x)=  \gamma(\beta) \int_0^\8 \int_{\rd} \rho\left(\sigma+t,x-z\right) \dB (\sigma,z) \dd \sigma \dd z.
\end{aligned}
\end{equation}
The limiting process $(\mathscr H(t,x))_{x\in \rd, t\geq 0}$ is a centered Gaussian field with covariance \begin{equation}\label{cov:H}
\mathrm{Cov}\big(\mathscr H(t,x),\mathscr H(s,y)\big)=\gamma^2(\beta)\int_0^\infty \rho(2\sigma +t+s,y-x) \dd \sigma\;.
\end{equation}
and $\mathscr H(t,x)$ is a real-valued solution of  the {\it non-noisy} heat equation 
\begin{equation} \label{eq:HE}
\partial_t \mathscr H=\frac 12 \Delta\mathscr H,
\end{equation}
but with {\it random initial condition} given by the Gaussian free field $\mathscr H(0,\cdot)$\footnote{The value $\mathscr H(0,x)$ at time $0$ is defined through the formula \eqref{eq:th1} for $t=0$.} with covariance structure 
\begin{equation}\label{cov:H:GFF}
\mathrm{Cov}\big(\mathscr H(0,x),\mathscr H(0,y)\big)=\gamma^2(\beta)\int_0^\infty \rho(2\sigma,x-y) \dd \sigma= \frac{\gamma^2(\beta)\Gamma(\frac d2-1)}{\pi^{d/2}|x-y|^{d-2}}.
\end{equation}
As a particular case of \eqref{eq:th1} we have that for any $x\in \rd$ and $t>0$, as $\e \to 0$,
\begin{equation} \label{th:h:EW}
 \mathscr H_\e(t,x)  \cvlaw \mathcal N\bigg(0,\gamma^2(\beta) \frac 2 {d-2} \frac 1{(4\pi)^{d/2}} t^{-(d-2)/2}\bigg).
\end{equation}
\end{theorem} 
We underline the following implication of Theorem \ref{th:h}. Note that a priori, the stationary solution $\hh$ itself does not admit an explicit statistical description (although it is a well-defined function evaluated on the noise, after time-reversal, rescaling and spatial translation). 
%(its law can only be interpreted in terms of the almost-sure limit as $T\to\infty$ of the free energy $\log\sZ_T$). 
However, Theorem \ref{th:h} shows that if the stationary solution $\hh$ is subtracted from its non-stationary counterpart $h_\e$ and the difference is rescaled by $\e^{1-d/2}$, then the limiting object (as $\e\downarrow 0$) is a {\it Gaussian free field} smoothened by the heat kernel. Thus, as a space-time process, the error term is analytically more tractable than the limiting (stationary) distribution itself (see also Theorem \ref{th:CVagainstTestFun} and Proposition \ref{prop:tightness} below). % for the convergence of the spatial averages $\int \dd x \varphi(x) \mathscr H_\e(t,x)$ and the tightness of the process $\big(\mathscr H_\e(t,x) \big)_{x\in \rd, t>0}$ in a certain function space). 

Next we turn to the fluctuations {\it of} the stationary solution $\hh$. Although it is statistically constant, $\hh(t,x)$ itself fluctuates around its mean $\E(\hh(t,x))$ and in this vein, our next main result identifies the rescaled fluctuation $\e^{1-\frac d2} (\hh(t,x)- \E[\hh(t,x)])$ in a spatially averaged sense. The limiting object is identified as a {\it stationary solution} of the stochastic heat equation with additive noise, but with {\it Gaussian free field marginals} (instead of flat initial condition). Here is the precise statement of this assertion, which is our next main result. 

\begin{theorem}[Fluctuations {\it of} the stationary solution]\label{thm:stationary}
 Assume $d\geq 3$. Then there exists $\beta_0\in (0,\b_c)$ such that for $\beta<\beta_0$ and for every smooth function $\varphi\in \mathcal C^\infty_c(\rd)$ with compact support and as $ \e \to 0$
 $$
 \e^{1-\frac d2}\int \dd x \varphi(x) [\hh(t,x)- \E(\hh(t,x))] \cvlaw \int \dd x \varphi(x) \HH(t,x) 
 $$
 with $\HH(\cdot,\cdot)$ being a stationary solution of the stochastic heat equation 
 \begin{equation}\label{StatEW}
 \partial_t \HH(t,x)= \frac 12 \Delta \HH(t,x)+ \overline{\gamma}(\beta) \xi, \qquad \HH(0,x)=  \mathscr H(0,x),
 \end{equation}
 with additive noise and with marginal distributions $\mathscr H(0,x)$ which is a Gaussian free field with covariance structure \eqref{cov:H:GFF} defined in Theorem \ref{th:h}, while
 $\overline\gamma^2(\beta)= \int_{\rd} \dd x V(x) E_x[\exp\{\beta^2\int_0^\infty V(W_{2s}) \dd s\}]$. 
 \end{theorem}
 
We also point out to the very interesting and independent recent works \cite{MU17,DGRZ18} where the rescaled and spatially averaged fluctuations $\e^{1-\frac d 2}\int_{\rd} \varphi(x) \big( h_\e(t,x)- \E[h_\e(t,x)]\big) \, \dd x$ of the (non-stationary) solution $h_\e$ of the KPZ equation \eqref{eq:KPZe} with flat initial condition $h_\e(0,\cdot)=0$ have been obtained. The limiting object (in law as $\e\downarrow 0$) is identified as $\int_{\rd} \varphi(x) \overline{\mathscr H}(t,x) \, \dd x$ with $\overline{\mathscr H}$ solving
the stochastic heat equation with additive noise, but with flat initial condition (in contrast to \eqref{StatEW}) -- that is, for $\beta\in(0,\beta_0)$, 
\begin{equation}\label{EW}\begin{aligned}
&\e^{1-\frac d 2}\int_{\rd} \varphi(x) \big( h_\e(t,x)- \E[h_\e(t,x)]\big) \, \dd x \cvlaw \int_{\rd} \varphi(x) \overline{\mathscr H}(t,x)\\
&\partial_t \overline{\mathscr H}= \frac 12 \Delta \overline{\mathscr H}+ \overline{\gamma}(\beta) \xi,\qquad \overline{\mathscr H}(0,x)=0. 
\end{aligned}
\end{equation} 
The above result is quite orthogonal in nature to Theorem \ref{th:h} (as well as the methods of proof, see Section \ref{sec:proofsketch} for a sketch). The centering in Theorem \ref{th:h} is done {\it about } the stationary solution and a space-time limiting process is obtained. In contrast, the deterministically centered fluctuations $h_\e(t,x)- \E[h_\e(t,x)]$ do not converge to zero pointwise (compare to \eqref{LLN}) but spatial averaging causes oscillations to cancel and a meaningful limit is obtained for the averages after rescaling by $\e^{1-\frac d2}$. Note that there is a qualitative difference also between \eqref{EW} and Theorem \ref{thm:stationary}: the former provides fluctuations of the  non-stationary solution $h_\e(t,x)$ (its distribution oscillates quite a bit as $\e>0$, $t>0$ and $x\in \rd$ vary) around its mean and the limiting object solves the additive noise equation \eqref{EW} with flat initial condition, while the latter provides the fluctuation of the stationary solution $\hh(t,x)$ (its law remains fixed even if the variables $\e>0$, $t>0$ and $x\in \rd$ oscillate) around its mean which converges to a stationary solution of \eqref{StatEW} with the (random) Gaussian free field initial condition. The link between these two results is built precisely  by another implication of Theorem \ref{th:h} -- the latter, combined with other tools developed on the way currently, implies that the limit $\HH$ in Theorem \ref{thm:stationary} manifests as an independent sum of the two limits $\mathscr H$ from Theorem \ref{th:h} and $\overline{\mathscr H}$ from \eqref{EW} (see Section \ref{proof:stationary} for details).

As remarked earlier, following up on Theorem \ref{th:h} we record two further results. The first is concerned with the convergence of the integrals $\int_{\rd} \dd x \varphi(x) \, \mathscr H_\e(t,x)$ against test functions $\varphi$:

\begin{theorem}[Spatially averaged fluctuations about the stationary solution]\label{th:CVagainstTestFun}
For $\beta\in (0,\beta_0)$ as in Theorem \ref{th:h} and for any smooth function $\varphi\in\mathcal{C}_c^\infty$ with compact support and any $t>0$, as $\e\to 0$, 
\begin{equation}
\int_{\mathbb{R}^d} \mathscr H_\e(t,x) \varphi(x) \rmd x \cvlaw \int_{\mathbb{R}^d} \mathscr H(t,x) \varphi(x) \rmd x\;.
\end{equation}
\end{theorem}
Next, as a byproduct of the arguments constituting the proof of Theorem \ref{th:h} and Theorem \ref{th:CVagainstTestFun}, we also obtain a tightness property of the 
process $\{\mathscr H_\e(t,x)\}_{\e>0,x\in \rd}$ which holds in a function space $\mathcal C^\alpha(\rd)$ of distributions with ``local $\alpha$-regularity" for $\alpha<0$, see 
Section \ref{sec-tightness} for a detailed definition. While this result is not used in the proof  
of Theorem \ref{th:CVagainstTestFun} 
it may be of independent interest. 
\begin{proposition}[Tightness] \label{prop:tightness}
For any $\beta \in (0,\beta_0)$ as in Theorem \ref{th:h} and $t>0$, the family $\{\mathscr H_\e(t,x)\}_{\e>0, x\in \rd}$ is tight in the function space $\mathcal C^\alpha(\rd)$ for all $\alpha < -\frac d 2$. Together with Theorem \ref{th:CVagainstTestFun}, we conclude that $\{\mathscr H_\e(t,x)\}_{\e>0, x\in \rd}$ converges to $\{\mathscr H(t,x)\}_{x\in \rd}$ in $\mathcal C^\alpha(\rd)$ for all $\alpha < -\frac d 2$.
\end{proposition}

\begin{rem}[The stationary solution]\label{rem:stationary}
\textup{For conceptual clarity, it is convenient to 
relate the smoothened KPZ solution $h_\e$ to the so-called {\it{free energy of the continuous directed polymer}}. 
For that, let $u_\e= \exp[h_\e]$ 
denote the {\it{Hopf-Cole solution}} of 
the linear multiplicative noise  stochastic heat equation (SHE) given by
\begin{equation}\label{eq:SHEe}
 \frac{\partial}{\partial t} u_{\e} = \frac12 \D u_{\e} +
 \b \e^{\frac{d-2}2} u_\e \, \dB_{\e} \;,\qquad\,\,  u_{\e}(0,x) =1,
\end{equation}
where the stochastic integral above is interpreted in the classical It\^o-Skorohod sense and we choose 
\begin{equation}\label{C-eps}
C_\e= \b^2(\phi\star \phi)(0) \e^{-2}/2= \b^2 V(0) \e^{-2}/2. 
\end{equation} 
Then, by the Feynman-Kac  formula \cite[Theorem 6.2.5]{K90} we have 
\begin{equation}\label{FK}
u_{\e}(t,x)=E_{x} \bigg[ \exp\bigg\{\beta\e^\frac{d-2}{2} \,\int_0^t \int_{\rd} \, \phi_\e(W_{
{ t-s}}-y)  \dB(s, y)\dd s \dd y -  t C_\e\bigg\}\bigg]\;.
\end{equation} 
With $\dB^{\ssup{\e,t,x}}(s,y)$ defined in \eqref{scaling} and plugging the latter in \eqref{FK}, and using Brownian scaling and time-reversal, we get the a.s.\  equality
\begin{equation}\label{eq:uZ}
u_\e(t,x)= \mathscr Z_{\frac t{\e^2}} \left(\xi^{\ssup{\e,t,0}}; \frac x \e\right)\;,
\end{equation}
where (recall \eqref{eq:Z}) $\mathscr Z_T(x)= \mathscr Z_T(\xi;x)$ is the martingale corresponding to the {\it normalized partition function} of the {\it continuous directed polymer} in a white noise environment $\xi$, while $\log \mathscr Z_T$ now stands for its aforementioned {\it{free energy}}. For $\beta\in (0,\beta_c)$,  
$\mathscr Z_T$ converges almost surely to a strictly positive limit $\sZ_\infty>0$ with a non-degenerate law and if $\mathfrak u(\xi)$ is any arbitrary representative of $\mathscr Z_\infty$,  then, letting $\mathfrak h=\log \mathfrak u$ and $\mathfrak h_\e(t,x) = \mathfrak h(\xi^{\ssup{\e,t,x}})$,
 we have the validity of \eqref{LLN} by \eqref{eq:uZ}. 
 Also, as remarked earlier, the sequence $\{\mathfrak h_\e(t,x)\}_{\e>0}$ 
 is {\it{constant in law}} for all $\e, t, x$, with the  law determined by that of $\log \mathscr Z_\infty$ (see Proposition \ref{prop:covar} for the spatial correlation structure of $\sZ_\infty$).
%Note that $\h it depends on $\e$, it is not a strong limit, but it can be used similarly which allows one to study the rescaled fluctuations $\e^{1-\frac d 2}[h_\e(t,x)- \mathfrak h_\e(t,x)]$.
Finally, $\hh(t,x)$ is a stationary solution of the regularized KPZ equation \eqref{eq:KPZe} with initial condition $\mathfrak h_\e(0,x) = \mathfrak h (\xi^{(\e,0,x)})$. The latter fact can be observed using the self-consistent property \eqref{eq:markov} and the Feynman-Kac formula.}\qed
\end{rem}

\begin{rem}[The limit $\mathscr H(t,x)$]\label{remark-H}
\textup{Interestingly, the evolution of $\mathscr H(t,x)$ obtained in Theorem \ref{th:h} is purely deterministic after having fixed the initial condition. Indeed,
 $$
 \begin{aligned}
\mathscr H(t,x)&=  \gamma(\beta) \int_0^\8 \rho\left(\sigma+t,\cdot\right) \star \dB (\sigma,\cdot)(x) \dd \sigma 
&= \left(\rho\left(t,\cdot\right) \star  \gamma(\beta)\int_0^\8 \rho\left(\sigma,\cdot\right) \star \dB (\sigma,\cdot) \dd \sigma\right) (x) \\
&= \rho\left(t,\cdot\right) \star  \mathscr H(0,\cdot) (x).
\end{aligned}
$$
The Gaussian process $\mathscr H(t,x)$ is, to our best knowledge, new in this context and in the literature.
The idea behind its deterministic evolution is that,
by the Feynman-Kac formula, the difference  $\mathscr H_\e(t,\cdot)$ can be viewed as an average under a polymer measure on $[0,t]$ of some centered field depending on the driving noise. By diffusivity, the averaging measure converges to Wiener measure, whereas the field depends linearly at first order on the white noise on negative-time  interval %$]-\8,0]$ 
coming from the stationary solution $\hh$. 
The limit is a Gaussian free field evolved by the heat equation.}\qed
\end{rem}

%\begin{rem}\label{remark1-Psi}
%\textup{While we are intrinsically interested in the asymptotic behavior of the rescaled deviation $\mathscr H_\e(t,x)$ for the KPZ solution $h_\e$, 
%the same strategy of our proofs works in more general situation. Indeed, if $u_\e$ denotes the solution of \eqref{eq:SHEe} and $\beta\in (0,\beta_0)$,
% then for any function $\Psi\in \mathscr C^1(\mathbb R; \mathbb R)$ with $\Psi(0)=0$ and $\Psi^\prime(0)=1$,  
%we can obtain convergence of the finite dimensional distributions of the spatially indexed process: as $\e \to 0$,
%\begin{equation} \label{eq:gnrlFunctionalCV}
%\e^{-\frac{(d-2)}2}\, \, \bigg(\Psi\bigg(\frac{u_\e(t,x)}{\mathfrak u(\xi^{\ssup{\e,t,x}})}-1 \bigg)\bigg)_{x\in \rd, t>0} \cvfidi  (\mathscr H(t,x))_{x\in \rd, t>0}.
%\end{equation}
%See also Remark \ref{remark2-Psi}.} \qed% which alludes to the main strategy of the proof. 
%\end{rem} 

\begin{rem}
\textup{For results similar to \eqref{EW} in $d=2$, see \cite{CD18,CSZ18}. By mollifying the noise in both time and space
related questions for the stochastic heat equation have been studied. This set up is quite different from the present one and leads to homogenized diffusion coefficients, i.e. in the limit $\frac 12 \D$ is replaced by $\frac 12 \mathrm{div}(\mathrm a_{\beta}\nabla)$ where $\mathrm a_\beta\ne \mathrm I_{d\times d}$ (see \cite{M17,GRZ17,DGRZ18h}).}\qed%Though being totally independent, our results share {\textcolor{black}{some}} common features with \cite{DGRZ18h}, which considers the stochastic homogeneization point of view for the heat equation in a random potential (a space-time function). Again, the space dimension is $d\geq3$ and the potential is assumed to be small. The authors derive a (random) corrector, i.e., a pointwise approximation of the solution  up to second order  as the ratio of the scale of variation of the initial condition to the correlation length of the noise vanishes. Here also,  taking a random centering is crucial. 
\end{rem}

 \subsection{On the method of proof.}\label{sec:proofsketch}
To provide some guidance to the reader, it is perhaps useful to underline some key ideas of the current method of the proof, which are 
quite orthogonal to the techniques used in the existent literature for showing Gaussian fluctuations of the KPZ equation. \footnote{Indeed, for proving the fluctuations \eqref{EW} of the non-stationary solution, \cite{DGRZ18} 
uses the integral representation $u_\e(t,x)= \E[u_\e(t,x)] + \beta \e^{(d-2)/2} \int_0^t\int_{\rd} \rho(t-s,x-y) u_\e(s,y) \, \xi_\e(s,y) \dd s \dd y$ for the stochastic heat equation \eqref{eq:SHEe}. Then the (rescaled and averaged) fluctuations  $\e^{1-d/2} \int_{\rd} \dd x  \varphi(x) [u_\e(t,x)- \E u_\e(t,x)]=\beta \int_0^t \int_{{\mathbb R}^{2d}} (P_{t-s}\varphi)(z-\e y) u_\e(s, z- \e y)  \varphi(y)  \xi(s,z) \dd s \dd y \dd z$ is written in terms of an integral representation w.r.t. the heat semigroup $(P_t)_t$, and due to asymptotic independence of $u_\e(t,x)$ and $u_\e(t,y)$ (for $x\ne y$ and as $\e\to 0$) and due to spatial averaging w.r.t. $\varphi$, the (r.h.s) in the integral representation converges to a centered Gaussian law.  An alternative method has been used in \cite{MU17} based on intricate applications of renormalization theory.} Also, the method developed here, for which we heavily exploit tools from stochastic analysis, is quite different from invoking stable and mixing convergence as in \cite{HL15,CL17,CN19}.

An important element of the present technique is to identify the limiting distributions as well as the correlation structure of the space-time process 
$T^{\frac{d-2}4}\big(\log \mathscr Z_{tT}(x\sqrt T)- \log \mathscr Z_\infty(x\sqrt T)\big)_{x\in \rd, t>0}$ (cf. Theorem \ref{th:Z} later). 
The main task for this purpose proceeds in four main steps.

\noindent{\bf Step 1:} The first step is to decompose $\log \sZ_T$, using It\^o's formula and write %$\log \sZ_T = N_T - \frac 12 \langle N\rangle_T$ where 
\begin{equation}\label{dec}
\begin{aligned}
&\log \sZ_T = N_T - \frac 12 \langle N\rangle_T, \qquad \mbox{where }\qquad
N_T=\beta \int_0^T \int_{\R^d} E_{0,\b,t} [ \phi (y-W_t) ]  \xi(t,y) \dd y \dd t,\,\,\,\mbox{and } \\
&\qquad\qquad \langle N\rangle_T = %\frac{\b^2}{2} 
\b^2
\int_0^T E_{0,\b,s}^{\otimes 2}  \left[ V(W_s^{\ssup 1}-W_s^{\ssup 2})  \right]  \dd s,
\end{aligned}
\end{equation}
where  $E_{0,\b,T}$ (resp.$E_{0,\b,T}^{\otimes 2}$) stands for the expectation w.r.t. the polymer measure (resp. the product of two independent polymer measures, cf. \eqref{polym-meas}). Note that $(N_T)_T$ is a martingale and 
 a key technical step is to prove the following decay of the derivative of its angled bracket 
\begin{equation}\label{decayN}
\frac{\dd}{\dd T}\langle N\rangle_T \sim \fC_0 T^{-d/2} \qquad\mbox{in probability.}
\end{equation}
for $\beta\in (0,\beta_0)$ and for an explicit constant 
 %For notational convenience we now use a constant 
\begin{equation}\label{C0}
\mathfrak C_0=\mathfrak C_0(\beta)= \gamma^2(\beta)\frac{1}{(4\pi)^{d/2}}\,,
\end{equation}
where $\gamma^2(\b)$ is given in \eqref{def-sigma-beta}. Note that we can express $N_T$ as well as its bracket in terms of 
 $\mathscr Z_T= E_0[\Phi_T(W)]$ with 
\[
 \dd\langle \sZ\rangle_T= \b^2E_0^\otimes\big[\Phi_T(W^{\ssup 1}) \Phi_T(W^{\ssup 2}) V(W^{\ssup 1}_T- W^{\ssup 2}_T)\big]\dd T \;,
\]
 where, for fixed Brownian path $W$, 
$ \Phi_T(W)= \exp\left\{\b \int_0^T \int_{\rd} \phi(W_s-y) \dB(s,y) \dd s\, \dd y - \frac{\b^2T}{2 }V(0)\right\}$
 is another martingale. Now we can compute 
\begin{eqnarray} \nn
&\dd \sZ_T=%& \b E_0 \bigg[ \Phi_T(W)\,\, \int_{\rd} \phi(y-W_T) \dB(T, y)  \, \dd y \bigg]\dd T = 
\b E_0 \bigg[ \Phi_T(W)\,\ (\phi \star \xi) (T,W_T)\bigg]\ \dd T \;,\\ \label{eq:differentialM}
&\dd \langle  \sZ\rangle_T = \b^2 E_0^{\otimes 2} \bigg[ \Phi_T(W^{\ssup 1})\Phi_T(W^{\ssup 2}) V\big(W_T^{\ssup 1}-W_T^{\ssup 2}\big)  \bigg]  \dd T \\ \label{eq:bracketM}
&=  \b^2 \sZ_T^{2} \times E_{0,\b,T}^{\otimes 2}  \big[ V(W_T^{\ssup 1}-W_T^{\ssup 2})  \big]  \dd T =\sZ_T^2 \dd\langle N\rangle_T\;,
\end{eqnarray}
so that the desired decay \eqref{decayN} follows once it is verified that for $\beta\in (0,\beta_0)$, 
 \begin{equation}\label{keystep}
 T^{d/2} \bigg(\frac \dd {\dd t} \langle \sZ\rangle\bigg)_T - \fC_0 \sZ_T^2 \stackrel{L^2(\mathbb P)}\to 0.
 \end{equation}
  Substantial technical work is needed to establish this step. %which, loosely speaking, is based on proving asymptotic independence 
% captured by the inherent attractive nature of the polymer measure.  

\noindent{\bf Step 2:} In the next step we need to quantify the decay of correlation of $\mathscr Z_T(x)$, leading to the estimate %which also provides the justification of the scaling factor $T^{\frac{d-2}4}$ in 
%$T^{\frac{d-2}4} [ \log \sZ_T(x)- \log \sZ_\infty(x)]$. The desired  is 
 \begin{equation}\label{correlation}
 \begin{aligned}
& \mathrm{Cov}\big(\mathscr Z_\infty(0), \mathscr Z_\infty(x)\big)= \fC_1\bigg(\frac 1 {|x|}\bigg)^{d-2} \quad\forall |x|\geq 1, \qquad \mbox{and}\\
& \|\mathscr Z_\infty - \mathscr Z_T\|^2_{L^2(\mathbb P)} \sim \fC_1\fC_2 \, T^{-\frac{(d-2)}2} \E[\mathscr Z_\infty^2],
 \end{aligned}
 \end{equation}
 for constants $\fC_1,\fC_2$ depending on $\beta$.  This decorrelation estimate is helpful to quantify the speed of convergence in \eqref{LLN}. Next, to
 identify the error of the approximation, the first step is to show that the {\it ratio} 
 \begin{equation}\label{eq0}
T^{\frac{d-2}4}\bigg( \frac{{\sZ_T(x)- \sZ_\infty(x)}}{\sZ_T(x)}\bigg) 
\end{equation} 
actually converges to a non-trivial distribution. For this pursuit, 
 we now introduce a sequence of 
 processes 
\[
 G^{\ssup T}=(G^{\ssup T}_\tau)_{\tau\geq 1}\qquad\mbox{with}\qquad G^{\ssup T}_\tau=T^{\frac{d-2}4} \left(\frac{\sZ_{\tau T}}{\sZ_T}-1\right),
\]
 and observe that, for each $T$, $G^{\ssup T}$ is a continuous martingale. We can compute its bracket: 
 \begin{equation}\label{bracket}
 \begin{aligned}
  \langle G^{\ssup T}\rangle_\tau= \frac{T^{\frac{d-2}2}}{\sZ_T^2} \langle \sZ\rangle_{\tau T} &=  \frac{T^{\frac{d-2}2}}{\sZ_T^2} \int_1^{\tau T} \bigg(\frac {\dd}{\dd t}\langle \sZ\rangle\bigg)_{s} \dd s \\
  &= \frac 1 {\sZ_T^2} \int_1^\tau (\sigma T)^{d/2} \bigg(\frac {\dd}{\dd t}\langle \sZ\rangle\bigg)_{\sigma T} \, \sigma^{-d/2}\dd \sigma.
  \end{aligned}
  \end{equation} 
 
 \noindent{\bf Step 3:} The next task is to justify that \eqref{bracket} and \eqref{keystep} imply 
 the convergence of the angled brackets $\{\langle G^{\ssup T}\rangle_\tau\}_T$ of the continuous martingales $G^{\ssup T}$  to a {\it deterministic limit}: 
 $$
 \langle G^{\ssup T}\rangle_\tau \to \fC_0 \int_1^\tau \sigma^{-\frac d 2} \, \dd\sigma =: g(\tau)
 $$  
Then it follows that the sequence $G^{\ssup T}$ itself converges in law to a Brownian motion with time-change given by $g$, i.e., $G^{\ssup T}\cvlaw G$ in $\mathscr C([1,\infty);\R)$ where $G$ is a mean-zero Gaussian process with independent increments and variance $g(\tau)= \frac 2 {d-2} \fC_0 [1- \tau ^{-\frac{(d-2)}2}]$.\footnote{It is well-known that a continuous martingale $M$ with a deterministic bracket $\langle M\rangle =\phi$ with $\phi:\R_+\to \R_+$ being continuous and increasing, is a process with independent and centered Gaussian increments. Moreover, if $\{M^{\ssup n}\}_n$ is a sequence of continuous martingales such that the sequence $\langle M^{\ssup n}\rangle$ converges in probability to a deterministic function $\phi$ as above, then $M^{\ssup n}$ itself converges in law to a process with independent, centered Gaussian increments. Based on a (conditional) second moment method, but using instead much of the process structure, this  argument is an efficient way to prove asymptotic normality, see e.g.   \cite{CN95}.}  
 Then,
 $$
 \lim_{T\to\infty} T^{\frac{d-2}4}\, \bigg(\frac {\sZ_\infty}{\sZ_T} - 1\bigg)= \lim_{T\to\infty}\lim_{\tau\to\infty} \bigg[G^{\ssup T}_\tau+ \frac{T^{\frac{(d-2)}4}[\sZ_\infty- \sZ_{\tau T}]}{\sZ_T}\bigg]
 $$
 and thanks to \eqref{correlation} and the fact that $\sZ_T\stackrel{\mathrm{a.s.}}\to \sZ_\infty>0$, the second term vanishes in the double limit, so that the left hand side converges to 
 the commutative limit $\lim_{T\to\infty}\lim_{\tau\to\infty} G^{\ssup T}_\tau= \lim_{\tau\to\infty} \lim_{T\to\infty} G^{\ssup T}_\tau$ which has a centered Gaussian law with variance $\frac 2 {d-2} \fC_0 (\beta) = g(\infty)$. The last assertion then also implies that the term \eqref{eq0} now converges to a centered Gaussian law with variance $\frac 2 {d-2}\mathfrak C_0$.

\noindent{\bf Step 4:}  We can now follow the recipe given in the above steps to show that the bracket of the rescaled martingale $N^{\ssup T}: \tau \mapsto T^{\frac{d-2}4} (N_{\tau T}- N_T)$ converges in probability to the deterministic function $g(\tau)$. Therefore, $N^{\ssup T} \cvlaw G$ and since the bracket $T^{\frac{d-2}4} \big(\langle N \rangle_{\tau T}- \langle N\rangle_T\big) \stackrel{\mathbb P}\to 0$ converges, the bracket term in the decomposition \eqref{dec} also vanishes. This idea now allows one to appeal to a multidimensional version of martingale CLT in an appropriate manner.  The actual machinery for showing convergence of the space-time process $(\mathscr H_\e(t,x))_{t>0,x\in \rd}$ for Theorem \ref{th:h} is more intricate, but it builds on an extension of the above guiding philosophy, which, together with further technical estimates developed on the way, is also helpful for proving Theorem \ref{thm:stationary} and Theorem \ref{th:CVagainstTestFun}.  %  existence of negative moments of the free energy $\log\mathscr Z_T$ proved in our earlier paper \cite{CCM19}.  %combined with a criterion which we appeal to from \cite{FM17}. 
\medskip

\noindent{\it Organization of the article:} The rest of the article is organized as follows. Section \ref{sec-proof-theorem} is devoted to proving the convergence of $\mathscr H_\e(t,x)$ for fixed $x\in \rd$ and $t>0$. One technical step for its proof is the aforementioned $L^2(\mathbb P)$ convergence \eqref{keystep} which constitutes Section \ref{sec:prclaim1}. The proof of the space-time convergence of Theorem \ref{th:h} as well as that of its averaged version in Theorem \ref{th:CVagainstTestFun}
are given in Section \ref{sec:proof:th:h}. Section \ref{proof:stationary} and Section \ref{sec-tightness} contain the proofs of Theorem \ref{thm:stationary} and that of Proposition \ref{prop:tightness}, respectively.

\section{Fluctuations about the limiting free energy.}\label{sec-proof-theorem}

The convergence results in 
Theorem \ref{th:h} and Theorem \ref{th:CVagainstTestFun} are directly linked with the following behavior of the free energy $\log\sZ_T$ of the continuous directed polymer, recall \eqref{eq:Z}. 
\begin{theorem}[Space-time fluctuations of the free energy]\label{th:Z} 
Under the assumptions imposed in Theorem \ref{th:h} 
we have, as $T\to\infty$, 
\begin{itemize}
\item  \begin{equation} \label{eq1:th:Z}
T^{\frac{d-2}4}\bigg(\log \mathscr Z_{tT}\big(x\sqrt T\big)- \log \mathscr Z_\infty\big(x\sqrt T\big)\bigg)_{x\in \rd, t>0} \cvfidi \big(\mathscr H(t,x)\big)_{x\in \rd, t>0}\,, 
\end{equation}
\item and for any $\varphi\in \mathcal C^\infty_c(\rd)$, 
\begin{equation} \label{eq2:th:Z}
T^{\frac{d-2}4}\int_{\rd} \dd x \, \varphi(x) \, \big[\log \mathscr Z_{tT}\big(x\sqrt T\big)- \log \mathscr Z_\infty\big(x\sqrt T\big)\big] \to \int_{\rd}\dd x \, \varphi(x)\, \mathscr H(t,x).
\end{equation}
\end{itemize}
\end{theorem}
In order to derive Theorem \ref{th:Z}, in this section we will derive the following assertions:

\begin{theorem}\label{co:tclMT}
Fix $d\geq 3$ and $\b < \b_0$ as in Theorem \ref{th:h}. Then for all $x\in \rd$, and as $T\to\infty$, 
\begin{equation}\label{eq1}
T^{\frac{d-2}4}\left(\log {\sZ_T(x)-\log \sZ_\infty(x)}\right)  \cvlaw \mathcal  N\big(0,2(d-2)^{-1}\,\fC_0\big),
\end{equation}
where $\fC_0(\beta)$ is defined above in \eqref{C0}. Moreover, for all $x\in \rd$, $t>0$, as $\e\to 0$,
\begin{equation}\label{eq2}
\mathscr H_\e(t,x)
 \cvlaw \mathcal  N\big(0,2(d-2)^{-1}\,\fC_0\,t^{-\frac{d-2}{2}}\big) \;.
\end{equation}
\end{theorem}

\begin{theorem}\label{th:CVmarginalZ}
Fix $d\geq 3$ and $\b < \b_0$ as in Theorem \ref{th:h}. Then for all $x\in \rd$, and as $T\to\infty$, 
\begin{equation}\label{eq:CVmarginalZ}
T^{\frac{d-2}4}\bigg( \frac{{\sZ_T(x)- \sZ_\infty(x)}}{\sZ_T(x)}\bigg)  \cvlaw \mathcal  N\left(0,\frac 2 {d-2} \fC_0 \right),
\end{equation}
where $\mathfrak C_0(\b)$ is defined in \eqref{C0}.

%Moreover, $\e^{-\frac{d-2}{2} } \left( \frac{ u_{\e} (t,x) }{\mathfrak u ( \dB^{\ssup{\e,t,x}} )} -1 \right)$ converges in law to the same limit.
\end{theorem}

The rest of this section is devoted to the proof of Theorem \ref{co:tclMT} and  Theorem \ref{th:CVmarginalZ}. We would like to emphasize that 
 although both theorems are easily shown to be  equivalent (see Remark \ref{remark2-Psi}), in Section \ref{subsec:secondProof} we provide a \emph{direct} proof of Theorem \ref{co:tclMT} which is in the same spirit as the one of Theorem \ref{th:CVmarginalZ}, but further relies on It\^o's decomposition of $\log \sZ_T$. As this idea turns out to be of much help in order to move to convergence against a test function (cf. Theorem \ref{th:CVagainstTestFun}) -- for which the argument of Remark \ref{remark2-Psi} {\it does not apply} -- we present here both proofs. This latter argument will also be repeatedly used throughout Section \ref{sec:proof:th:h}. 

%\textcolor{cyan}{I would maybe move here the "Central idea of the proof [of Theorem \ref{th:CVmarginalZ}]"}

\subsection{Rate of decorrelation.}\label{subsec:rateOfDec}
%%%%
In this section we will provide the following elementary result, which provides an estimate on the asymptotic decorrelation of $u_\e(t,x)$ and $u_\e(t,y)$ as $\e\to 0$ through identification \eqref{eq:uZ}. This estimate also underlines the fact that smoothing $u_\e(x)$ w.r.t. 
any $\varphi\in \mathcal C_c^\infty(\rd)$ makes $\int_{\rd} \dd x \, u_\e(t,x) \, \varphi(x)\to \int_{\rd} \dd x \, \overline u(t,x) \varphi(x) $ deterministic, with $\partial_t \overline u=\frac 12 \Delta \overline u$.
%%%%
\begin{proposition}\label{prop:covar}  Let $d\geq 3$ and $\beta$ small enough.
% $\b < \sup\bigg\{ b>0: E_0\left[ e^{b^2 \int_0^\8 V(\sqrt 2 W_s) ds}\right] <\8\bigg\} \in (0, \8)$ \textcolor{red}{Don't we take $\b$ smaller than this ?}. \\
\begin{itemize}
\item We have: 
%\begin{equation} \nn
%{\rm Var}\big( M_T(x)\big)= E_0\left[ e^{\b^2 \int_0^T V(\sqrt 2 W_s) ds} -1 \right]
%\end{equation}
%\item % With $H_A= \inf\{ s \geq 0: W_s \in A\}$, and 
\begin{equation}\label{eq:covM0}
{\rm Cov}\big(  \sZ_\8(0),  \sZ_\8(x)\big)=
\begin{cases}
 E_{x/\sqrt{2}} \bigg[ \mathrm e^{\b^2 \int_0^\8 V(\sqrt 2 W_s) ds} -1 \bigg] \quad \forall x\in \rd, \\
%\\ \label{eq:covM}
%& = & P_{x}(H_{\6 B(0,1)}<\8) \times E_{\sqrt 2\eu}\left[ e^{\b^2 \int_0^\8 V(\sqrt 2 W_s) ds} -1 \right] \qquad {\rm if}  \; |x| \geq 2, \\ \nn
\fC_1  
% \Big(\frac{1}{|x|}\Big)^{d-2}    \ \,\qquad\qquad\qquad\qquad\forall\, |x| \geq 1,
|x|^{-(d-2)}    \ \,\qquad\qquad\qquad\qquad\forall\, |x| \geq 1
\end{cases}
\end{equation}
with $\fC_1= E_{\eu/\sqrt 2}[\exp\{\b^2 \int_0^\8 V(\sqrt 2 W_s) ds\} -1]$.\footnote{the constant $1$ in the second case of the covariance is related to the assumed correlation length of the noise, recall that $V=\phi\star \phi$ has support within a ball of radius $1$ around the origin.}
\item Finally,
\begin{eqnarray} 
\big\| \sZ_\8- \sZ_T\big\|_2^2  \label{eq:asdifnorm}
&\sim& \fC_1 \fC_2 \E\big[\sZ_\8^2\big] T^{-\frac{d-2}{2}}
 \qquad {\rm as} \; T \to \8.
\end{eqnarray}
with $\fC_2= E\big[ \big({\sqrt 2}/|Z|\big)^{d-2} \big]  $, where $Z$ is a centered Gaussian vector with covariance $I_d$.
\end{itemize}
\end{proposition}
%%%%%%%%%%
\begin{rem}
The random process $(\fC_1^{-1/2} \sZ_\8(x); x \in \rd)$ is quite interesting. It has the same covariance as the Gaussian free field, see \eqref{eq:covM0} for $|x|\geq 1$.  It is a  positive stationary field with covariance equal to the Green function  (at distance 1 from the diagonal). We stress that, in this range of values for $x$, the covariance does not  not depend on the regularizing function $\phi$, whereas the law does depend. \qed
\end{rem}
%%%%%%%%%%%
\begin{proof}
For any Brownian path $W=(W_s)_{s\geq 0}$ we set 
\begin{equation} \label{eq-Phi}
\Phi_T(W)= \exp\left\{\b \int_0^T \int_{\rd} \phi(W_s-y) \dB(s,y) \dd s\, \dd y - \frac{\b^2T}{2 }V(0)\right\}
\end{equation}
and see that, for any $n\in \mathbb N$, %From the identity $\E \exp G =  \exp \frac 12 \E G^2$ for a centred  Gaussian r.v. $G$, we have
\begin{equation} \label{eq:expgauss}
\E \bigg[   \prod_{i=1}^n \Phi_T( W^{\ssup i} ) \bigg] =
\exp \bigg\{ \b^2 \int_0^{T} \sum_{1 \leq i < j \leq n} V\big( W_s^{\ssup i}- W_s^{\ssup j} \big) ds  \bigg\}.
\end{equation}
 Then, the first line of equation \eqref{eq:covM0} follows from \eqref{eq:expgauss} with $n=2$ and the fact that $W_s^{(1)}-W_s^{(2)}\eqlaw \sqrt 2 W_s$. 
Now, the second line of  \eqref{eq:covM0}  follows by considering the hitting time of the unit ball for $\sqrt{2}W$ and spherical symmetry of $V$.

We now show \eqref{eq:asdifnorm} as follows. For two independent paths $W^{\ssup 1}$ and $W^{\ssup 2}$  (which are also independent of the noise $\dB$), we will denote by $\mathcal F_T$ the $\sigma$-algebra generated by 
both paths until time $T$. Then, by Markov property, for $0<S\leq \8$,
\begin{equation}\label{eq:markov}
\sZ_{T+S}(x) = E_x\big[ \Phi_T(W) \,\,
\sZ_S \circ \tht_{T,W_T}
\big]\;.
\end{equation}
where for any $t>0$ and $x\in \rd$, $\theta_{t,x}$ denotes the canonical spatio-temporal shift in the white noise environment. Hence, 
\begin{eqnarray*}
\|\sZ_\8-\sZ_T\|_2^2  &=& \E \bigg[E_0^{\otimes 2}\bigg\{ \Phi_T(W^{\ssup 1}) \Phi_T(W^{\ssup 2})  \left(\sZ_\8 \circ \theta_{T,W_T^{\ssup 1}}-1\right)\left(\sZ_\8  \circ \theta_{T,W_T^{\ssup 2}}-1\right)  \bigg\}\bigg]\\
 &=& E_0^{\otimes 2}\bigg[ \mathrm e^{\b^2 \int_0^{T} V( W_s^{\ssup 1}-W_s^{\ssup 2}) \dd s} \times {\rm Cov}\big(\sZ_\8(W_T^{\ssup 1}),\sZ_\8(W_T^{\ssup 2})\big)  \bigg]\\
&=& E_0^{\otimes 2}\bigg[ \mathrm e^{\b^2 \int_0^{T} V( W_s^{\ssup 1}-W_s^{\ssup 2}) \dd s} \times  E_0^{\otimes 2}\bigg[  {\rm Cov}\big(\sZ_\8(W_T^{\ssup 1}),\sZ_\8(W_T^{\ssup 2})\big)  \bigg \vert \mathcal F_T\bigg]\bigg] \\
&\sim& \fC_1  E_0^{\otimes 2}\left[ \mathrm e^{\b^2 \int_0^{T} V( W_s^{\ssup 1}-W_s^{\ssup 2}) \dd s} \times \left(\frac{ 2}{|W_T^{\ssup 1}-W_T^{\ssup 2}|}\right)^{d-2} \right]  %\qquad\ ({\rm by\ }\eqref{eq:covM}) 
\\
&=& \fC_1  
 E_0\left[ \mathrm e^{\b^2 \int_0^T V(\sqrt 2 W_s) \dd s}  \bigg(\frac{\sqrt 2}{|W_T|}\bigg)^{d-2} \right] 
 %\\   &\sim& \fC_1 E_0^{\otimes 2}\left[ e^{\b^2 \int_0^{T} V( W_s-\tW_s) ds} \right]  E_0^{\otimes 2}\left[ \left(\frac{ 2}{|W_T-\tW_T|}\right)^{d-2} \right] 
% \qquad ({\rm asymptotic\ indep.})
\end{eqnarray*}
The asymptotic equivalence in the above display is justified if we note that for $\xi^{\ssup t}(s,y)=\xi(s+t,y)$,
\begin{equation}\label{eq:covM}
{\rm Cov}\big(  \sZ_\8(0;\xi),  \sZ_\8(x,\xi^{\ssup t})\big)=
\int_{\rd} \rho(t,y)  E_{(y-x)/\sqrt{2}} \bigg[ \mathrm e^{\b^2 \int_0^\8 V(\sqrt 2 W_s) ds} -1 \bigg] \dd y
\end{equation}

Then \eqref{eq:asdifnorm} is proved once we show 
\begin{equation} \label{eq:asymptoticindependence1}
 E_0\bigg[ \mathrm e^{\b^2 \int_0^T V(\sqrt 2 W_s) \dd s}  \bigg(\frac{1}{|W_T|}\bigg)^{d-2} \bigg]
 \sim E_0\bigg[ \mathrm e^{\b^2 \int_0^{\8} V(\sqrt 2 W_s) \dd s} \bigg]  E_0\bigg[ \left(\frac{1}{|W_T|}\right)^{d-2} \bigg] 
\end{equation}
But as $T\to\infty$, 
\begin{equation} \nn
\left( \int_0^{T} V(\sqrt 2 W_s) \dd s, T^{-1/2}W_T \right)  \stackrel{\rm law}{\longrightarrow} \left( \int_0^{\8} V(\sqrt 2 W_s) \dd s, Z \right) 
\end{equation}
with $Z\sim N(0,I_d)$ being independent of the Brownian path $W$, and then \eqref{eq:asdifnorm} follows from the requisite uniform integrability 
\begin{equation}\label{eq-ui}
\sup_{T \geq 1}  E_0\left[\left(\mathrm e^{\b^2 \int_0^T V(\sqrt 2 W_s) \dd s}  \left(\frac{T^{1/2}}{|W_T|}\right)^{d-2} \right)^{1+\delta}\right] < \8 
\end{equation}
for $\delta>0$. By H\"older's inequality and Brownian scaling, for any $p,q\geq 1$ with $1/p+1/q=1$, 
$$
\begin{aligned}
\mbox{(l. h. s.) of \eqref{eq-ui}} &\leq E_0\bigg[\mathrm e^{q(1+\delta)\b^2 \int_0^T V(\sqrt 2 W_s) \dd s}\bigg]^{1/q}\,\, E_0\bigg[ \frac 1 {|W_1|^{p (1+\delta) (d-2)}} \bigg]^{1/p} \\
&\leq E_0\bigg[\mathrm e^{q(1+\delta)\b^2 \int_0^\infty V(\sqrt 2 W_s) \dd s}\bigg]^{1/q} \,\, \bigg[\int_{\mathbb R^d} \dd x \,\, \frac 1 {|x|^{p(1+\delta)(d-2)}} \mathrm e^{-|x|^2/2}\bigg]^{1/p} \\
&\leq C \bigg[\int_0^\infty \dd r  \,\, r^{d-1}\,\, \frac 1 {r^{p(1+\delta)(d-2)}} \mathrm e^{-r^2/2}\bigg]^{1/p} 
\end{aligned}
$$
Then the last integral is seen to be finite provided we choose $\delta>0$ and $p>1$ small enough so that $1< p(1+\delta) \leq \frac {d}{d-2}$.  
%\textcolor{red}{On pourrait prendre $< d/(d-2)$}
 \end{proof}
 We will also need the following uniform $L^p$-estimate whose proof is based on Proposition \ref{prop:covar} as well as uniform estimates on moments of $\log\mathscr Z_T$ proved in \cite{CCM19}.  
 
 \begin{proposition} \label{prop:LpBoundLogZ}
For all $1<p<2$,
\begin{equation} \label{eq:lpbound}
\sup_{T>0}\E\left[\left|T^{(d-2)/4}(\log\sZ_T-\log\sZ_\infty)\right|^p\right] < \infty.
\end{equation}
\end{proposition}
\begin{proof}
Write $\log \sZ_T - \log \sZ_\infty$ as $\log (1- \frac{\sZ_\infty - \sZ_T}{\sZ_\infty})$ and decompose the LHS of \eqref{eq:lpbound} as:
\begin{align*}
&\E\left[\left|T^{(d-2)/4}\left(\log \sZ_T - \log\sZ_\infty\right)\right|^p \mathbf{1}\left(\left|\frac{\sZ_\infty - \sZ_T}{\sZ_\infty}\right| > \frac 12\right)\right]\\
& + \E\left[\left|T^{(d-2)/4}\log \left(1- \frac{\sZ_\infty - \sZ_T}{\sZ_\infty}\right)\right|^p \mathbf{1}\left(\left|\frac{\sZ_\infty - \sZ_T}{\sZ_\infty}\right| \leq \frac 12 \right)\right],
\end{align*}
which defines the sum of two terms $A_T+B_T$. For the second term, we can use the upper bound $|\log(1-x)|\leq C|x|$ which holds for all $x\leq 1/2$, and, combined with H\"older's inequality with $1/(2/p)+1/q=1$, entails
\begin{equation} \label{eq:boundBTterm}
B_T \leq C \E\left[\left|T^{(d-2)/4} \left(\frac{\sZ_\infty - \sZ_T}{\sZ_\infty}\right)\right|^p \right]\leq C \E\left[T^{(d-2)/2} \left(\sZ_\infty - \sZ_T\right)^2 \right]^{p/2}\E\left[\sZ_\infty^{-pq}\right]^{1/q}.
\end{equation}
In \cite{CCM19}, it is proven that $\sZ_\infty$ admits all negative moments for all $\b<\b_{L^2}$. Moreover $T^{(d-2)/2}(\sZ_\infty - \sZ_T)^2$ is bounded in $L^1$ by \eqref{eq:asdifnorm}, hence $\sup_T B_T$ is finite. 

Then, we use the upper bound $|\alpha + \mu|^p \leq 2^p(|\alpha|^p + |\mu|^p)$ and H\"older's inequality with $1/a + 1/b = 1$, choosing $b>1$ small enough so that $bp<2$, and obtain:
\begin{align*}
A_T & \leq  T^{p(d-2)/4 } \, 2^p\left(\E\left[|\log Z_T|^{pa}\right]^{1/a} + \E\left[|\log Z_\infty|^{pa}\right]^{1/a}\right) \IP\left(\left|\frac{\sZ_\infty - \sZ_T}{\sZ_\infty}\right| > \frac 12\right)^{1/b}\\
& \leq C T^{\frac{p}{2}\frac{(d-2)}{2}} \E\left[ \left|\frac{\sZ_\infty - \sZ_T}{\sZ_\infty} \right|^{pb}\right]^{1/b}\\
& \leq C \E\left[T^{(d-2)/2} \left(\sZ_\infty - \sZ_T\right)^2 \right]^{p/2},
\end{align*}
where the second inequality comes 
\begin{equation}\label{eq:negative}
\sup_T \E\left[|\log\sZ_T|^\theta\right]<\infty, \qquad \forall\theta\in \mathbb R,
\end{equation}
which has been shown for $\beta \in (0,\beta_0)$ in \cite[Theorem 1.3]{CCM19}\footnote{In \cite[Theorem 1.3]{CCM19} the existence of negative (and positive) moments is stated for $\log\mathscr Z_\infty$. However, exactly the same proof yields a uniform (in $T$) estimate for $\log \mathscr Z_T$.} and Markov's inequality, and where the third inequality follows from the upper-bound \eqref{eq:boundBTterm} with $p$ replaced by $pb$ and again invoking \eqref{eq:negative}. As above, this implies that $\sup_T A_T$ is finite, which ends the proof.
\end{proof}

\subsection{Proof of Theorem \ref{th:CVmarginalZ}}   
 We start by computing the stochastic differential and bracket of the martingale $\sZ_T$ defined as follows: % {\bf [Replace with the second notation]}
\begin{eqnarray} \nn
&\dd \sZ_T=%& \b E_0 \bigg[ \Phi_T(W)\,\, \int_{\rd} \phi(y-W_T) \dB(T, y)  \, \dd y \bigg]\dd T = 
\b E_0 \bigg[ \Phi_T(W)\,\ (\phi \star \xi) (T,W_T)\bigg]\ \dd T \;,\\ \label{eq:differentialM}
&\dd \langle  \sZ\rangle_T = \b^2 E_0^{\otimes 2} \bigg[ \Phi_T(W^{\ssup 1})\Phi_T(W^{\ssup 2}) V\big(W_T^{\ssup 1}-W_T^{\ssup 2}\big)  \bigg]  \dd T \\ \label{eq:bracketM}
&=  \b^2 \sZ_T^{2} \times E_{0,\b,T}^{\otimes 2}  \left[ V(W_T^{\ssup 1}-W_T^{\ssup 2})  \right]  \dd T \;,
\end{eqnarray}
where $E_{0,\b,T}^{\otimes 2}$ is the expectation taken with respect to the product of two independent polymer measures, 
\begin{equation}\label{polym-meas}
P_{0,\beta,T}(\dd W^{\ssup i})=\frac 1 {Z_{\beta,T}} \,\, \exp\bigg\{ \beta \int_0^T \int_{\rd} \phi(W_s^{\ssup i}-y) \,\,\dB(y, s) \dd s \dd y \bigg\} \,\, P(\dd W^{\ssup i})\qquad i=1,2,
\end{equation}
with $ {Z_{\beta,T}}= e^{-\frac{\beta^2}{2}TV(0)}\sZ_T $.
The proof of Theorem \ref{th:CVmarginalZ} splits into two main steps. The first step involves showing the following estimate whose proof constitutes Section \ref{sec:prclaim1}: %we will prove the key estimate
%%%%%%
 \begin{proposition}\label{prop:claim1} There exists $\b_0\in(0,\infty)$, such that for all $\b < \b_0$, as $T \to \8$,
\begin{equation}\nn
T^{\frac{d}{2}} \left(\frac{\dd}{\dd t} \langle  \sZ \rangle\right)_T -  \fC_0(\b) \sZ_T^2 \stackrel{L^2}{\longrightarrow} 0, 
\end{equation}
where $\mathfrak C_0(\b)$ is defined in \eqref{C0}.
% Other notations
%$$
%T^{\frac{d}{2}} \frac{d}{dt} \langle  M \rangle (T) -  \fC_3 M(T)^2 
%$$
%$$
% = T^{\frac{d}{2}} \frac{d \langle  M \rangle }{dt}(T)
% $$
%with  $\$ defined in \eqref{eq:fC_4}. %\textcolor{blue}{I think we have forgotten the $\b^2$ in the definition $\fC_3$ / $\sigma^2(\b)$, which comes from \eqref{eq:bracketM}}
\end{proposition}
%%%%%%%%%
For the second step, we define a sequence $\{G^{\ssup T}_\tau\}_{\tau\geq 1}$ of stochastic processes on time interval $[1, \8)$, with 
\begin{equation}\label{def:GT}
%G_T(\tau)=  T^{\frac{d-2}{4}}  \left( \frac{M(\tau T)}{M(T)} -1 \right) \;,\qquad \tau \geq 1.
G^{\ssup T}_\tau=  T^{\frac{d-2}{4}}  \left( \frac{\sZ_{\tau T}}{\sZ_T} -1 \right) \;,\qquad \tau \geq 1.
\end{equation}
Then, for all $T$, $G^{\ssup T}$ is a continuous martingale for the filtration ${\mathcal B}^{\ssup T}=({\mathcal B}^{\ssup T}_\tau)_{\tau \geq 1}$, where ${\mathcal B}^{(T)}_\tau$ denotes
the $\sigma$-field generated by the white noise $\dB$ up to time $\tau T$. Then we need the following result, which 
provides convergence at the process level:
%%%%%%%%%%%%
\begin{theorem} \label{th:cvGP}
For $\b < \beta_0$, as $T \to \8$, we have convergence
\begin{equation} 
G^{\ssup T} \stackrel{\rm law}{\longrightarrow} G
 \end{equation}
on the space of continuous functions on $[1, \8)$, equipped with the topology of uniform 
convergence on compact intervals, where $G$ is a mean zero Gaussian process with independent increments  and variance
\begin{equation} \label{eq:VarianceLimitMartingale}
g(\tau) = \frac{2}{d-2}\fC_0(\beta) \,\, [1-\tau^{-\frac{d-2}{2}}].
\end{equation}
\end{theorem}
\paragraph{Proof of Theorem  \ref{th:cvGP} (Assuming Proposition \ref{prop:claim1}):}  From the definition \eqref{def:GT} we compute the bracket 
of the square-integrable martingale $G^{\ssup T}$,
%$$
%G^{(T)}_\tau = \frac{T^\frac{d-2}{2}}{M_T^{2}}\langle M \rangle_{\tau T}
%= \frac{T^\frac{d-2}{2}}{M_T^2} \int\limits_0^{\tau T} 
%\left(\frac{d}{dt} \langle  M \rangle\right)_{s} d s
%= \frac{T^\frac{d}{2}}{ \sZ_T^2} {\displaystyle  \int\limits_0^\tau }
%\left(\frac{d}{dt} \langle  M \rangle\right)_{\sigma T} d \sigma 
%$$
$$
\begin{aligned}
\langle G^{\ssup T}\rangle_\tau = \frac{T^\frac{d-2}{2}}{\sZ_T^{2}}\,\, \langle \sZ \rangle_{\tau T}&= \frac{T^\frac{d-2}{2}}{\sZ_T^2} \int_1^{\tau T} 
\left(\frac{\dd}{\dd t} \langle  \sZ \rangle\right)_{s} \dd s
\\  
&= \frac{T^\frac{d}{2}}{\sZ_T^2} {\displaystyle  \int_1^\tau }
\left(\frac{\dd}{\dd t} \langle  \sZ \rangle\right)_{\sigma T} \dd \sigma 
\end{aligned}
$$
by replacing the variables $s=\sigma T$. Then,
$$
\begin{aligned} 
\langle G^{\ssup T}\rangle_\tau -g(\tau)
& = {\displaystyle  \int_1^\tau }
\bigg[  \frac{(\sigma T)^\frac{d}{2}}{\sZ_T^2}
\left(\frac{\dd}{\dd t} \langle  \sZ \rangle\right)_{\sigma T} 
-  \fC_0\bigg]
\sigma^{-d/2} \dd \sigma 
\\ 
&= {\displaystyle  \int_1^\tau }
\frac{\sZ_{\sigma T}^2}{\sZ_{ T}^2}
\left[  \frac{(\sigma T)^\frac{d}{2}}{\sZ_{\sigma T}^2}
\left(\frac{\dd}{\dd t} \langle  \sZ \rangle\right)_{\sigma T} 
- \fC_0\right]
\sigma^{-d/2} \dd \sigma 
+
 \frac{ \fC_0}{\sZ_T^2} 
{\displaystyle  \int_1^\tau }
\left[ \sZ_{\sigma T}^{2}\!-\!{\sZ_T^2}\right]
\sigma^{-d/2} d \sigma 
\\
& =: I_1 + I_2
\end{aligned}
$$
As $T \to \8$ the last integral vanishes in $L^1$ and $I_2$ vanishes in probability. 
For $\varepsilon \in (0,1]$, introduce the event 
$$
A_\varepsilon= \bigg\{ \sup\big\{\sZ_t; t\in[0,\infty]\big\}\vee\sup\big\{\sZ_t^{-1}; t\in[0,\infty]\big\} \leq \varepsilon^{-1}\bigg\}
$$
and observe that $\lim_{\varepsilon \to 0} \IP( A_\varepsilon) =1$ since $ \sZ_t$ is continuous, positive with a positive limit.  So, we can estimate the expectation of $I_1$ by
$$
\E \big[ {\bf 1}_{ A_\varepsilon}  \vert I_1 \vert
 \big]
\leq \frac{\tau}{\varepsilon^6} \,\, \sup_{ t \geq T}\bigg\{ \bigg \| 
t^{\frac{d}{2}} \left(\frac{\dd}{\dd t} \langle  \sZ \rangle\right)_t \!-\!   \fC_0 \sZ_t^2 
\bigg\|_1 \bigg\}  ,
$$
which vanishes by Proposition \ref{prop:claim1}. Thus, $\langle G^{\ssup T} \rangle \to g$ in probability. 
Since for the sequence of continuous martingales $G^{\ssup T}$  the brackets converge pointwise 
to a deterministic limit $g$, we derive that the sequence  $G^{\ssup T}$ itself converges in law to a Brownian motion with time-change given by $g$, that is, the process $G$ defined in the statement of Theorem \ref{th:cvGP} (see \cite[Theorem 3.11 in  Chapter 8]{JS87}), which is proved now.
\qed

\paragraph{Concluding the proof of Theorem \ref{th:CVmarginalZ}.} 
Write
\begin{eqnarray*}
T^{\frac{d-2}{4}}  \left( \frac{\sZ_\8}{\sZ_T} -1 \right)  &=& G_\8^{\ssup T}\\&=& G_\tau^{\ssup T} +\frac{ T^{\frac{d-2}{4}} [\sZ_\8-\sZ_{\tau T}] }{\sZ_T} \;,
\end{eqnarray*}
and consider the last term. By \eqref{eq:asdifnorm}, the numerator has $L^2$-norm tending to 0 as $\tau \to \8$ uniformly in $T \geq 1$ whereas the denominator has a positive limit.
Then, the last term vanishes in the double limit $T \to \8, \tau \to \8$, and therefore 
$$
\lim_{T \to \8} T^{\frac{d-2}{4}}  \left( \frac{\sZ_\8}{\sZ_T} -1 \right) = \lim_{\tau \to \8} \lim_{T \to \8} G^{\ssup T}(\tau) ,
$$
which is the Gaussian law with variance $g(\8)\!=\frac{2}{d-2}\fC_0$ by Theorem \ref{th:cvGP}. Hence,
$$
T^{\frac{d-2}4}\bigg(\frac{\sZ_T(x)-\sZ_\infty(x)}{\sZ_T(x)}\bigg)  \cvlaw \mathcal N\left(0,\frac{2}{d-2}\fC_0\right)\,\,\, \mbox{as}\,\,T\to\infty
$$
\qed

\subsection{Proof of Theorem \ref{co:tclMT}.} \label{subsec:secondProof}
Note that the convergences \eqref{eq1} and \eqref{eq2} are equivalent by identification \eqref{eq:uZ}, we therefore only prove \eqref{eq1}. 
%Moreover, \eqref{eq:uZ} together with Theorem \ref{th:CVmarginalZ} implies that 
%\begin{equation}\label{eq::2}
%\e^{-\frac{d-2}{2} } \left( \frac{\mathfrak u( \dB^{\ssup{\e,t,x}} )}{ u_{\e} (t,x) } -1 \right) 
% \cvlaw \mathcal N\left(0,2(d-2)^{-1}\fC_0\,t^{-\frac{d-2}{2}}\right) \quad\mbox{as} \,\, \e\to 0,\mbox{      for all}\,\,\,\, t>0, x\in \rd.
%\end{equation} 
By It\^o's formula, equations \eqref{eq:differentialM} and \eqref{eq:bracketM} imply that 
\begin{equation}\label{eq:claim}
\log \sZ_T = N_T - \frac 12 \langle N\rangle_T,
\end{equation}
where
$$
\begin{aligned}
&N_T=\beta \int_0^T \int_{\R^d} E_{0,\b,t} [ \phi (y-W_t) ]  \xi(t,y) \dd y \dd t \\
& \langle N\rangle_T = %\frac{\b^2}{2} 
\b^2
\int_0^T E_{0,\b,s}^{\otimes 2}  \left[ V(W_s^{\ssup 1}-W_s^{\ssup 2})  \right]  \dd s,
\end{aligned}
$$
with $N$ being a martingale.
Then, Proposition \ref{prop:claim1} shows that, in probability and as $T\to\infty$,
\[T^{d/2}\frac{\dd}{\dd T}\langle N\rangle_T \to \fC_0.\] 

Mimicking the proof of Theorem \ref{th:cvGP}, we further get that the bracket of the rescaled martingale $N^{\ssup T}:\tau\to T^{(d-2)/4}(N_{\tau T}-N_T)$ converges in probability as $T\to\infty$ to the deterministic function $g(\tau)$, where $g$ is given by \eqref{eq:VarianceLimitMartingale},  implying thereby the convergence $N^{\ssup T}\cvlaw G$.
Moreover, convergence of the bracket also implies that as $T\to\infty$,
\[
T^{(d-2)/4}(\langle N \rangle_{\tau T} - \langle N \rangle_{T}) {\stackrel{\IP}{\longrightarrow}}\, 0,
\]  
so that the bracket part in \eqref{eq:claim} vanishes under the scaling limit.
Putting things together, we obtain that for all $\tau \geq 1$,
\begin{equation} \label{eq:CVforLog}
T^{(d-2)/4}(\log \sZ_{\tau T} - \log \sZ_{T}) \cvlaw G(\tau).
\end{equation}
Then, we write:
\begin{equation*}
T^{(d-2)/4}\left(\log \sZ_{\infty} - \log \sZ_{T}\right) = T^{(d-2)/4}(\log \sZ_{\infty} - \log \sZ_{\tau T})  + T^{(d-2)/4}(\log \sZ_{\tau T} - \log \sZ_{T})
,\end{equation*}
where, by Proposition \ref{prop:LpBoundLogZ}, the first term vanishes in $L^p$-norm as $\tau\to\infty$, uniformly in $T\geq 1$, for all $p<2$.
Therefore, convergence \eqref{eq1} follows from \eqref{eq:CVforLog} and the last display, by letting $T \to \infty$ and $\tau\to\infty$ as in the conclusion of the proof of Theorem \ref{th:CVmarginalZ}. 
\qed

\begin{rem}\label{remark2-Psi}
Note that with the following lemma, we can see that Theorem \ref{th:CVmarginalZ} and Theorem \ref{co:tclMT} are in fact equivalent (choose for example $\Psi(x)=\log(1+x)$ to go from Theorem \ref{th:CVmarginalZ} to Theorem \ref{co:tclMT}). We can also obtain a general version of the two theorems which is Corollary \ref{co:tclMT0}.
\begin{lemma} \label{lem:switchingLemma}
If $X_\e\to \mathcal N(0,\sigma^2)$ in distribution, then for any $\Psi\in C^1(\mathbb R)$ with $\Psi(0)=0$, \ $ \Psi^\prime(0)=1$ and $a>0$, we have
$\e^{-a} \Psi(\e^a X_\e)\to \mathcal N(0,\sigma^2)$. 
\end{lemma}
\begin{cor}\label{co:tclMT0}
Fix $d\geq 3$ and $\b < \b_0$ as in Theorem \ref{th:h} and $\Psi$ as in Lemma \ref{lem:switchingLemma}. Then for all $x\in \rd$, $t>0$, as $\e\to 0$,
\begin{equation} \label{eq1-2}
\e^{-\frac{(d-2)}2}\, \, \bigg(\Psi\bigg(\frac{u_\e(t,x)}{\mathfrak u(\xi^{\ssup{\e,t,x}})}-1 \bigg)\bigg) \cvlaw \mathcal N\left(0,\frac{2}{d-2}\fC_0t^{-\frac{d-2}{2}}\right).
\end{equation}
\end{cor}
\begin{proof}
The proof follows from Theorem \ref{th:CVmarginalZ}, identification \eqref{eq:uZ} and the above lemma. 
\end{proof}
%Note that Corollary \ref{co:tclMT0} implies follows directly from the second convergence in \eqref{eq::2} and the following lemma, with $a=(d-2)/2$:
\end{rem}

%%%%%%

%%%%%%%%%%%%
\section{Proof of Proposition \ref{prop:claim1}.} \label{sec:prclaim1}

This section is entirely devoted to the proof of Proposition \ref{prop:claim1}. Denote for short by $\mathscr L_T$ the quantity of interest,
\begin{eqnarray}
\mathscr L_T &:=& T^{\frac{d}{2}} \left(\frac{\dd}{\dd t} \langle  \sZ \rangle\right)_T -  \fC_0 \sZ_T^2\nn \\ \nn
&=&  E_{0}^{\otimes 2}  \bigg[\Phi_T(W^{\ssup 1}) \Phi_T(W^{\ssup 2}) \bigg(T^{\frac{d}{2}}  V\big(W_T^{\ssup 1}\!-\! W_T^{\ssup 2}\big) - \fC_0 \bigg)\bigg],
\end{eqnarray}
and proceed in two steps. The first step is devoted to determining the value of the appropriate constant $\fC_0$ below. 

%%%%%%%%%
\subsection{The first moment.}\label{sec:warmup}
We first want to show that:
\begin{proposition}\label{cor-LT}
There exists $\b_1\in(0,\infty)$ such that for all $\b< \b_1$, if we choose
\begin{equation} \label{eq:fC_4}
\mathfrak C_0(\b) = \frac{\b^2}{(2\pi)^{d/2}} \int_{\rd} \dd y \,\, V(\sqrt{2} y)\,\,E_y\bigg[\mathrm e^{\beta^2\int_0^\infty V(\sqrt{2} W_{s})\,\dd s}\bigg],
\end{equation}
then $ \E(\mathscr L_T)\to 0$ as $T\to\infty$.
\end{proposition}
\begin{rem}
By a simple change of variables, one can check that the definition of $\fC_0$ in \eqref{eq:fC_4} corresponds indeed to the one of Proposition \ref{prop:claim1}.
\end{rem}

The rest of Section \ref{sec:warmup} is devoted to the proof of Proposition \ref{cor-LT}. For any $t>s\geq 0$ and $x,y\in\rd$, we will denote by $P_{s,x}^{t,y}$ the law (and by 
$E_{s,x}^{t,y}$ the corresponding expectation) of the Brownian bridge starting at $x$ at time $s$ and conditioned to reach $y$ at time $t>s$. 
Recall that
$$
\rho(t,x)= (2\pi t)^{-d/2} \mathrm e^{-|x|^2/2t},
$$
denotes the standard Gaussian kernel. 
%Before considering second moment we choose the value of $\fC_3$ by making the first moment (without absolute value!) zero in the limit, i.e., such that

We note that 
\begin{eqnarray}\nn
\E (\mathscr L_T) &=& E_{0}^{\otimes 2}  \bigg[ 
\mathrm e^{\b^2\int_0^T V(W^{\ssup 1}_t\!-\!W^{\ssup 2}_t)\dd t} \bigg( T^{\frac{d}{2}}  V\big(W_T^{\ssup 1}\!-\! W_T^{\ssup 2}\big) - \fC_0 \bigg)\bigg]% {\longrightarrow} 0 
\\ \nn
&=&  E_{0}  \bigg[
\mathrm e^{\b^2\int_0^{T} V(\sqrt 2  W_t)\dd t} \bigg( T^{\frac{d}{2}}  V(\sqrt 2  W_T) - \fC_0 \bigg)\bigg],
%\\ \nn
% {\longrightarrow} 0 
\end{eqnarray}
and
\begin{equation}\nn
E_{0}  \bigg[  \mathrm e^{\b^2\int_0^{T} V(\sqrt 2  W_t)\dd t}T^{\frac{d}{2}}  V(\sqrt 2  W_T) \bigg] =
\int_{\rd}  V(\sqrt 2 y) E_{0,0}^{T,y}  \bigg[  \mathrm e^{\b^2\int_0^{T} V(\sqrt 2  W_t)\dd t} \bigg] T^{\frac{d}{2}}  \rho(T,y) \dd y.
\end{equation}
Now, we fix a sequence 
$$
m=m(T) \qquad\mbox{such that}\qquad m\to \infty \quad\mbox{and} \quad m=o(T)\quad\mbox{as}\quad T\to\infty,
$$
which allows us to prove Proposition \ref{cor-LT} in two steps:
\begin{proposition}\label{prop-claim1}
For small enough $\b$, as $T\to\infty$,
\begin{equation*}
\lim_{T\to\infty} \sup_{y\in\mathbb{R}^d} \left|E_{0,0}^{T,y}   \bigg[ 
\mathrm e^{\b^2\int_0^{T} V(\sqrt 2  W_t)\dd t} \bigg] - \mathcal T_1(y)\right|=0,
\end{equation*}
where
$$\mathcal T_1(y)=
E_{0,0}^{T,y}   \left[ 
\mathrm e^{\b^2\int_{[0,m] \cup [T\!-\!m,T]} V(\sqrt 2  W_t)\dd t} \right].
$$
\end{proposition}
\noindent 
%[The reason is that for small $\beta$ Brownian motion with these exponential weights remains transient.] 

\begin{proposition}\label{prop-claim2}
For small enough $\b$, for all fixed $y\in\mathbb{R}^d$ and as $T\to\infty$, 
\begin{eqnarray} \nn
%E_{0,0}^{T,y}   \left[  e^{\b^2\int_{[0,m] \cup [T\!-\!m,T]} V(\sqrt 2  W_t)\dd t}T^{\frac{d}{2}}  V(\sqrt 2  W_T) \right] 
\mathcal T_1(y)
&\sim &
E_{0,0}^{T,y}   \bigg[ 
\mathrm e^{\b^2\int_{[0,m] } V(\sqrt 2  W_t)\dd t} \bigg]
E_{0,0}^{T,y}   \bigg[ 
\mathrm e^{\b^2\int_{ [T\!-\!m,T]} V(\sqrt 2  W_t)\dd t} \bigg]\\ \nn
&\to&
E_{0}  \bigg[ \mathrm e^{\b^2\int_0^{\8} V(\sqrt 2  W_t)\dd t} \bigg]  E_{y}   \bigg[ 
\mathrm e^{\b^2\int_{0}^\8 V(\sqrt 2  W_t)\dd t}\bigg].
\end{eqnarray}
\end{proposition}
%[Some ideas could be grapped from \cite{Va04} and the proof of \eqref{eq:asymptoticindependence1}.]  

\medskip

We will provide some auxiliary results which will be needed to prove Proposition \ref{prop-claim1} and Proposition \ref{prop-claim2}.
First, we state a simple consequence of Girsanov's theorem:
\begin{lemma}\label{lemma1-claim1}
%Recall the notation $\rho(s,x)$ for the Brownian motion transition density function. Let $P_y$ be the Wiener measure of the Brownian motion starting at $y$, with $P_{0,y}^{t,z}$ being the law of the Brownian bridge conditioned on starting at $y$ at time $0$ and arriving at $z$ at time $t$. Then for any $s<t$, $z,y\in\mathbb{R}^d$,
For any $s<t$ and $y,z\in \rd$, the Brownian bridge $P_{0,y}^{t,z}$ is absolutely continuous w.r.t. $P_{0,y}$ on the $\sigma$-field $\mathcal F_{[0,s]}$ generated by the Brownian path until time $s<t$, and 
\begin{equation} \label{lemma1-claim1-densityBBBM}
\begin{aligned}
\frac{ \dd P_{0,y}^{t,z}}{\dd P_{0,y}}\bigg|_{\mathcal F_{[0,s]}} &=\ \frac{\rho(t-s,z-W_s)}{\rho(t,z-y)}
&\leq \bigg(\frac t {t-s}\bigg)^{d/2} \exp\bigg\{\frac {|z-y|^2}{2t}\bigg\}.
\end{aligned}
\end{equation}
\end{lemma}
\qed
We will need the following version of Khas'minskii's lemma \cite[p.8, Lemma 2.1]{S98} for the Brownian bridge: 
\begin{lemma}\label{lemma2-claim1}
If $E_0\bigg[2 \beta^2 \int_0^\infty V(\sqrt 2 W_s) \dd s\bigg] <1$, then 
%{\color{blue} 
%[Write (i) uniform estimate on starting point TOO; (ii) for BM, not only for bridge]}
\[
\sup_{z,x\in \mathbb R^d,t>0} E_{0,x}^{t,z}\bigg[\exp\bigg\{\beta^2 \int_0^t V(\sqrt 2 W_s) \dd s\bigg\}\bigg] <\infty.
\]
\end{lemma}
\begin{proof}
By Girsanov's theorem, for any $s<t$, $\alpha\in \R^d$ and $A\in \mathcal F_{[0,s]}$, 
\begin{equation}\label{eq1-lemma2-claim1}
P_{0,x}^{t,z}(A)= E^{(\alpha)}_x \bigg[ \frac{\rho^{(\alpha)}(t-s; z-W_s)}{\rho^{(\alpha)}(t,z-x)} \,\mathbf 1_A\bigg] 
\end{equation}
where $E^{(\alpha)}$ (resp. \unskip \  $P^{(\alpha)}$) refers to the expectation (resp. \unskip \ the probability) with respect to Brownian motion with drift $\alpha$ and transition density 
\[
\rho^{(\alpha)}(t,z)= \frac 1 {(2\pi t)^{d/2}} \exp\bigg\{- \frac{|z- t\alpha|^2}{2t}\bigg\}.
\]
With $\alpha=(z-x)/t$ and $s=t/2$, applying \eqref{eq1-lemma2-claim1}, we get 
\[
P_{0,x}^{t,z}(A) \leq 2^{d/2}\,\, P^{(\alpha)}_x (A).
\]
Replacing $A$ by $e^{2\beta^2 \int_0^{t/2} V(\sqrt 2 W_s) \dd s }$, we have
\[
\begin{aligned}
\sup_{z,x\in \mathbb R^d,t>0} E_{0,x}^{t,z}\bigg[\exp\bigg\{2\beta^2 \int_0^{t/2} V(\sqrt 2 W_s) \dd s\bigg\}\bigg] &\leq 2^{d/2} \sup_\alpha E^{(\alpha)}\bigg[\exp\bigg\{2\beta^2 \int_0^{t/2} V(\sqrt 2 W_s) \dd s\bigg\}\bigg] \\
&\leq 2^{d/2} \frac 1 {1-a} <\infty,
\end{aligned}
\]
where the second upper bound follows from Khas'minskii's lemma provided we have 
\[
2\beta^2 \sup_{x,\alpha} E^{(\alpha)}_x \bigg[ \int_0^\infty V(\sqrt 2 W_s) \dd s \bigg] \leq a <1.
\]
But since the expectation in the above display is equal to $\int_0^\infty \dd s \int_{\R^d} \dd z V(\sqrt 2 z)\,\, \rho^{(\alpha)} (s, z-x)$ and is maximal for $x=0$ and $\alpha=0$, the requisite condition reduces to 
\[
2\beta^2 E_0 \bigg[ \int_0^\infty V(\sqrt 2 W_s) \dd s \bigg]<1,
\]
which is satisfied by our assumption. Finally,  the lemma follows from the observation
 $$
 \exp\bigg\{\beta^2\int_0^t V(\sqrt 2 W_s) \,\dd s\bigg\}\leq \frac 12 \bigg[\exp\bigg\{2\beta^2\int_0^{t/2} V(\sqrt 2 W_s) \dd s \bigg\}+ \exp\bigg\{2\beta^2\int_{t/2}^{t} V(\sqrt 2 W_s) \dd s \bigg\}\bigg]
 $$
  combined with time reversibility of Brownian motion. 
\end{proof}
Recall that $V=\phi\star\phi$ is bounded and has support in a ball of radius $1$ around the origin, and therefore, for some constant $c, c^\prime>0$, and any $a>0$, 
$$
P_0\bigg[\int_m^\infty \dd s \, V(\sqrt 2 W_s) >a\bigg] \leq \frac {c}a \int_m^\infty \frac{\dd s}{s^{3/2}} \int_{B(0,1)} \dd y V(\sqrt 2 y) \exp\bigg\{-\frac{|y|^2}{2s}\bigg\} \leq \frac{c^\prime \|V\|_\infty} {am^{1/2}} \to 0 %\quad\mbox{as}\,\,
$$
as $m\to\infty$, implying
\begin{lemma}\label{lemma2.5-claim1}
For any $a>0$, $\lim_{T\to\infty}\,\, P_0\big[\int_m^\infty \dd s \, V(\sqrt 2 W_s) >a\big] =0$.
\end{lemma}
By Lemma \ref{lemma2-claim1}, we also have
\begin{lemma}\label{lemma3-claim1}
For any $a>0$, 
%{\color {blue} [Not only $m\to \8$ but also $T\to \8$. Settle : (i) the choice of $m(T)$ (ii)  $m<\!<T \iff T \geq 2m$; (iii) Idem in proof of Prop. 3.1]
%}
$$
\lim_{T\to\infty}\sup_{z\in \mathbb R^d} P_{0,0}^{T,z}\bigg[\int_m^{T-m} V(\sqrt 2 W_s) \dd s >a \bigg] =0.
$$
\end{lemma}

\begin{proof}[{\bf{Proof of Proposition \ref{prop-claim1}}}]
Note that it is enough to show that, for any $a>0$ and $y\in\mathbb{R}^d$,
\begin{equation} \label{eq:sufficientLimit}
\lim_{T\to\infty} \sup_{y\in\mathbb{R}^d} \, E_{0,0}^{T,y}  \bigg[ \mathrm e^{\b^2\int_0^{T} V(\sqrt 2  W_t)\dd t}\,\,\, \mathbf 1\bigg\{\int_m^{T-m} V(\sqrt 2 W_s)\,\dd s >a\bigg\}\bigg]  =0.
\end{equation}
Indeed, letting $A=\left\{\int_m^{T-m} V(\sqrt 2 W_s)\,\dd s >a\right\}$, we have
\begin{align*}
 \left|E_{0,0}^{T,y}   \bigg[ 
\mathrm e^{\b^2\int_0^{T} V(\sqrt 2  W_t)\dd t} \bigg] - \mathcal T_1\right| & = E_{0,0}^{T,y}   \left[ 
\mathrm e^{\b^2\int_{[0,m] \cup [T\!-\!m,T]} V(\sqrt 2  W_t)\dd t}\left(e^{\b^2\int_m^{T-m} V(\sqrt 2  W_t)\dd t}-1 \right)\right]\\
&\leq  2E_{0,0}^{T,y}   \bigg[ 
\mathrm e^{\b^2\int_0^{T} V(\sqrt 2  W_t)\dd t} \, \mathbf 1_A \bigg] + ae^aE_{0,0}^{T,y}   \bigg[ 
\mathrm e^{\b^2\int_0^{T} V(\sqrt 2  W_t)\dd t} \bigg],
\end{align*}
where we decomposed the RHS of the first line on the two events $A$ and $A^c$ and we used that $|e^a-1|\leq ae^a$ in the second line. From this last display, we see that Proposition \ref{prop-claim1} is obtained by choosing $a$ arbitrary small, Lemma \ref{lemma2-claim1} and \eqref{eq:sufficientLimit}.

Finally, we observe that convergence \eqref{eq:sufficientLimit} follows from H\"older's inequality,  Lemma \ref{lemma2-claim1} and Lemma \ref{lemma3-claim1}, which ends the proof.
%\qed.
%$$
%\begin{aligned}
%&\sup_{|y|\leq 2, m\ll T} \,\,E_{0,0}^{T,y}  \bigg[ e^{\b^2\int_0^{T} V(\sqrt 2  W_t)\dd t}\,\,\, \mathbf 1\bigg\{\int_m^{T-m} V(\sqrt 2 W_s)\,\dd s >a\bigg\}\bigg] \\
%& \leq \sup_{|y|\leq 2, m\ll T}\,\, \bigg[E_{0,0}^{T,y}  \bigg[ e^{p \b^2\int_0^{T} V(\sqrt 2  W_t)\dd t}\bigg]^{1/p} \,\, P_{0,0}^{T,y}\bigg\{ \int_m^{T-m} V(\sqrt 2 W_s) \,\, \dd s>a\bigg\}^{1/q} \bigg]\\
% &\leq C^{1/p} \sup_{|y|\leq 2, m\ll T} E_{0}\bigg[ \mathbf 1\bigg\{\int_m^{T-m} V(\sqrt 2 W_s) \,\, \dd s>a\bigg\} \frac{ \dd P_{0,0}^{T,y}}{\dd P_0}\bigg]^{1/q}.
% \end{aligned}
% $$
% Bounding the integral $\int_m^{T-m}$ separately into two parts $\int_m^{T/2}$ and $\int_{T/2}^{T-m}$ and invoking again Lemma \ref{lemma1-claim1} and Lemma \ref{lemma3-claim1}, we get for any $a>0$, some 
% $\e_a(m)\downarrow 0$ as $m\to\infty$, such that 
 %$$
% \sup_{|y|\leq 2, m\ll T}\,\, P_{0,0}^{T,y}\bigg\{ \int_m^{T-m} V(\sqrt 2 W_s) \,\, \dd s>a\bigg\} \leq \e_a(m),
% $$
% Proposition \ref{prop-claim1} is proved.
\end{proof}

We now turn to the proof of 

\begin{proof}[{\bf{Proof of Proposition \ref{prop-claim2}}}]
Condition on the position of the Brownian bridge at time $T/2$, then use reversal property of the Brownian bridge and change of variable $z\to\sqrt{T} z$, to get:
\begin{align*}
&\mathcal T_1(y) = \int_{\mathbb{R}^d} E_{0,0}^{T/2,z}\left[e^{\b^2\int_{[0,m] } V(\sqrt 2  W_t)\dd t} \right]E_{T/2,z}^{T,y}\left[\mathrm e^{\b^2\int_{[T-m,T] } V(\sqrt 2  W_t)\dd t} \right] \frac{\rho(T/2,z)\rho(T/2,y-z)}{\rho(T,y)}\dd z\\
& = \int_{\mathbb{R}^d} E_{0,0}^{T/2,z\sqrt{T}}\left[e^{\b^2\int_0^m V(\sqrt 2  W_t)\dd t} \right]E_{0,y}^{T/2,z\sqrt{T}}\left[\mathrm e^{\b^2\int_0^m V(\sqrt 2  W_t)\dd t} \right] \frac{\rho(1/2,z)\rho(1/2,z-y/\sqrt{T})}{\rho(1,y/\sqrt{T})}\dd z.
\end{align*}
We now claim that, for fixed $z$,
\begin{equation} \label{eq:BBToBM}
E_{0,y}^{T/2,z\sqrt{T}}\left[\mathrm e^{\b^2\int_0^m V(\sqrt 2  W_t)\dd t} \right] \sim E_{y}\left[\mathrm e^{\b^2\int_0^\infty V(\sqrt 2  W_t)\dd t} \right].
\end{equation}
Then, by dominated convergence theorem applied to the above integral, where the expectations in the integrand are bounded thanks to Lemma \ref{lemma2-claim1}, we obtain that:
\begin{align*}
\mathcal T_1(y) \sim & \int_{\mathbb{R}^d} E_0\left[\mathrm e^{\b^2\int_0^\infty V(\sqrt 2  W_t)\dd t} \right]E_y\left[\mathrm e^{\b^2\int_0^\infty V(\sqrt 2  W_t)\dd t} \right] \frac{\rho(1/2,z)\rho(1/2,z)}{\rho(1,0)}\dd z\\
& = E_0\left[\mathrm e^{\b^2\int_0^\infty V(\sqrt 2  W_t)\dd t} \right]E_y\left[\mathrm e^{\b^2\int_0^\infty V(\sqrt 2  W_t)\dd t} \right].
\end{align*}
To prove \eqref{eq:BBToBM}, we use Lemma \ref{lemma1-claim1}:
\begin{align*}
&E_{0,y}^{T/2,z\sqrt{T}}\left[\mathrm e^{\b^2\int_0^m V(\sqrt 2  W_t)\dd t} \right] \\
& = \frac{1}{\rho(T/2,z\sqrt{T}-y)} E_{y} \left[\mathrm e^{\b^2\int_0^m V(\sqrt 2  W_t)\dd t} \rho(T/2-m,z\sqrt{T}-\sqrt{2}W_m) \right]\\
& = \frac{1}{\rho(1/2,z-y/\sqrt{T})\left(\pi(1-\frac{2m}{{T}})\right)^{d/2}} E_{y} \left[\mathrm e^{\b^2\int_0^m V(\sqrt 2  W_t)\dd t} \mathrm e^{-\frac{|z-\sqrt{2/T}\,W_m|^2}{1-2m/T}}\right].
\end{align*}
By monotone convergence and the fact that $m=o(T)$, we obtain:
\[P\text{-a.s.}\quad 
\mathrm e^{\b^2\int_0^m V(\sqrt 2  W_t)\dd t} \to \mathrm e^{\b^2\int_0^\infty V(\sqrt 2  W_t)\dd t}\quad \text{and} \quad
\mathrm e^{-\frac{|z-\sqrt{2/T}\,W_m|^2}{1-2m/T}} \to \mathrm e^{-2z^2}.
\]
%\textcolor{blue}{For the first display we can bound $|e^x- e^y| \leq |x-y| [e^{x} + e^y]$, so that, by monotone convergence,}
%$$
%\textcolor{blue}{P \bigg[ \bigg |\int_0^m V(\sqrt 2 W_s) \,\, \dd s  - \int_0^\infty V(\sqrt 2 W_s) \,\, \dd s \bigg | \geq a \bigg] \leq \frac C a E \bigg[\int_0^\infty V(W_s) (1- \mathbf 1_{[0.m]})\dd s  \bigg] \,\, E\bigg[ e^{ \beta^2 \int_0^\infty V(W_s) \,\, \dd s}\bigg] \to 0.}
%$$
%\textcolor{blue}{Also, is not the second display immediate if we choose $m$ so that $m \leq O(\sqrt T)$ and we know $W_m/m\to 0$ a.s.?}
Then, we have the following uniform integrability property for small $\delta >0$ and small $\b$:
\[
E_y\left[\left(\mathrm e^{\b^2\int_0^m V(\sqrt 2  W_t)\dd t} \mathrm e^{-\frac{|z-\sqrt{2/T}\,W_m|^2}{1-2m/T}}\right)^{1+\delta}\right] \leq E_y\left[\mathrm e^{(1+\delta)\b^2\int_0^\infty V(\sqrt 2  W_t)\dd t}\right]<\infty.
\]
 Hence,
\begin{equation*}
E_{0,y}^{T/2,z\sqrt{T}}\left[\mathrm e^{\b^2\int_0^m V(\sqrt 2  W_t)\dd t} \right] \to \frac{\mathrm e^{-2z^2}}{\rho(1/2,z)\pi^{d/2}} E_{y} \left[\mathrm e^{\b^2\int_0^\infty V(\sqrt 2  W_t)\dd t} \right] = E_{y} \left[\mathrm e^{\b^2\int_0^\infty V(\sqrt 2  W_t)\dd t} \right].
\end{equation*}

\end{proof}

%%%%%%
\subsection{Second moment.}\label{sec:second}

The goal of this section is to show 

\begin{proposition}\label{prop-second-moment}
There exists $\b_0\in (0,\infty)$, such that for all $\b< \b_0$, $\E(\mathscr L_T^2) \to 0$. 
\end{proposition}

For this result we will proceed as in the proof of Proposition \ref{cor-LT}. It is enough to show that $\limsup_{T\to\infty} \mathbb E(\mathscr L_T^2)\leq 0$. Computing second moment, we get an integral over four independent Brownian paths: 
%coupled by ${{4}\choose{2}} =6$ interaction terms :
\begin{eqnarray}
\E(\mathscr L_T^2) &= E_{0}^{\otimes 4}   \bigg[ \prod_{i\in\{1,3\}} \bigg( T^{\frac{d}{2}}  V(W^{\ssup i}_T\!-\!W^{\ssup{i+1}}_T) - \fC_0 \bigg)  %\left[ T^{\frac{d}{2}}  V(W^{(3)}_T\!-\!W^{(4)}_T) - \fC_3 \bigg] 
\mathrm e^{\b^2 \sum_{1\leq i < j \leq 4} \int_0^T V(W^{\ssup i}_t\!-\!W^{\ssup j}_t)\dd t} \bigg] 
\nn \\
&= E_{0}^{\otimes 4} \bigg[  \; \prod_{i\in\{1,3\}} \bigg\{\mathrm e^{\b^2\int_0^T V(W^{\ssup i}_t\!-\!W^{\ssup{i+1}}_t)\dd t}  \bigg(T^{\frac{d}{2}}  V(W^{\ssup i}_T\!-\!W^{\ssup{i+1}}_T) - \fC_0\bigg) \bigg\}
\label{eq1-second} \\
%&& \qquad \qquad \left. \times\;
%e^{\b^2 \int_0^T V(W^{(3)}_t\!-\!W^{(4)}_t)\dd t}  \left[ T^{\frac{d}{2}}  V(W^{(3)}_T\!-\!W^{(4)}_T) - \fC_3 \right] \right. \\ \nn &&  \qquad \qquad \qquad \qquad \left. \times\;
&\qquad\times \mathrm e^{\b^2\sum^{*} \int_0^T V(W^{\ssup i}_t\!-\!W^{\ssup j}_t)\dd t}\; \bigg]
\nn
\end{eqnarray}
where the sum $\sum^{*} $ is considered for $4$ pairs $(i,j), {1\leq i < j \leq 4}$ different from $(1,2)$ and $(3,4)$.
\medskip

 %Our goal splits into two tasks, written in Proposition \ref{prop-claim3} and Proposition \ref{prop-claim4} as follows, with which we will also conclude that 
%$\mathbb E L_T^2\to 0$. 
Throughout the rest of the article, for notational convenience, we will write 
\begin{equation}\label{eq-H-m}
\begin{aligned}
&H_m= \mathrm e^{\b^2 \sum_{1\leq i < j \leq 4}  \int_0^m V(W^{\ssup i}_t\!-\!W^{\ssup j}_t)\dd t} \, , \\
&H_{T-m,T}=\prod_{i\in\{1,3\}}\bigg\{\mathrm e^{\b^2  \int_{T-m}^T V(W^{(i)}_t\!-\!W^{(i+1)}_t)\dd t}\,\,\bigg(T^{d/2} V\left(W^{(i)}_T\!-\!W^{(i+1)}_T\right) - \fC_0 \bigg)\bigg\}.
\end{aligned}
\end{equation}
We will now estimate each term in the expectation in \eqref{eq1-second}. 
Proposition \ref{prop-claim3} stated below enables us to neglect the contributions
of $\int_m^{T-m} V(W^{\ssup i}_t\!-\!W^{\ssup j}_t)\dd t $ for all $i,j$ %(see section \ref{sec:warmup}),
and of
$\int_{T-m}^{T} V(W^{\ssup i}_t\!-\!W^{\ssup j}_t)\dd t $ for all $(i,j) \neq (1,2), (3,4)$. 
More precisely, we want to show that 
\begin{proposition}\label{prop-claim3}
For $m={m(T)}$ as above, there exists a constant $C>0$ such that, for small enough $\b$, as $T\to\infty$,
\[
 \E \mathscr L_T^2 = \mathcal T_2 + o(1),  
\]
where
\begin{equation}\label{eq1-T2}
\begin{aligned}
 \mathcal T_2&=  E_{0}^{\otimes 4}   \big[  H_m \,\, H_{T-m,T}\big].%e^{\b^2 \sum_{1\leq i < j \leq 4}  \int_0^m V(W^{(i)}_t\!-\!W^{(j)}_t)\dd t} \\
 %&\qquad\qquad\qquad\times \prod_{i=1,3} \bigg[e^{\b^2\int_{T-m}^T V(W^{(i)}_t\!-\!W^{(i+1)}_t)\dd t} \bigg( T^{\frac{d}{2}}  V(W^{(i)}_T\!-\!W^{(i+1)}_T) - \fC_3 \bigg)\bigg]  \bigg]
 \end{aligned}
\end{equation}
\end{proposition}
Then, Proposition \ref{prop-second-moment} will be a consequence of 
\begin{proposition}\label{prop-claim4}
For small enough $\b$, we have as $T\to\infty$:
\begin{equation} \label{eq:prop-claim4}
\begin{aligned}
 \mathcal T_2 &=  E_{0}^{\otimes 4}   \bigg[ \mathrm e^{\b^2 \sum_{1\leq i < j \leq 4}  \int_0^\infty V(W^{(i)}_t\!-\!W^{(j)}_t)\dd t} \bigg] \\
 &\qquad\qquad\times \left[E_{0}^{\otimes 2}   \left(  \mathrm e^{\b^2\int_{T-m}^T V(W^{(1)}_t\!-\!W^{(2)}_t)\dd t} \left[ T^{\frac{d}{2}}  V(W^{(1)}_T\!-\!W^{(2)}_T) - \fC_0 \right] \right) \right]^2 + o(1).
 \end{aligned}
\end{equation}
As a result, $\mathcal T_2 \to 0$.
 \end{proposition}

\subsection{Proof of Proposition \ref{prop-claim3}.}\label{sec:proof-claim3}
By the proof of Proposition \ref{prop-claim1}, we see that it is enough to prove that
\begin{equation} \label{eq:removingProofEst}
\sup_{T>0} E_{0}^{\otimes 4} \bigg[ \mathrm e^{\b^2 \sum_{1\leq i < j \leq 4}  \int_0^T V(W^{\ssup i}_t\!-\!W^{\ssup j}_t)\dd t} 
 \prod_{i=\{1,3\}} \bigg|T^{d/2} V(W^{\ssup i}_T\!-\!W^{\ssup{i+1}}_T) - \fC_0 \bigg| \bigg] <\infty,
\end{equation}
and that for all $a>0$,
\begin{equation} \label{eq:removingProofEst2}
\lim_{T\to\infty}  E_{0}^{\otimes 4} \bigg[ \mathrm e^{\b^2 \sum_{1\leq i < j \leq 4}  \int_0^T V(W^{\ssup i}_t\!-\!W^{\ssup j}_t)\dd t} 
 \prod_{i=\{1,3\}} \bigg|T^{d/2} V(W^{\ssup i}_T\!-\!W^{\ssup{i+1}}_T) - \fC_0 \bigg| \, \mathbf 1_A\bigg] = 0,
\end{equation}
where the event $A$ is defined as 
\begin{align*}
A= \bigg\{\int_{m}^{T-m} & V(W^{\ssup i}_t\!-\!W^{\ssup j}_t)\dd t \geq a\,\,\mbox{for some}\,\, 1\leq i < j \leq 4,\bigg\}\\
&\bigcup\bigg\{ \int_{T-m}^T V(W^{\ssup i}_t\!-\!W^{\ssup j}_t)\dd t \geq a \mbox{ for some}\,\, (i,j) \neq (1,2), (3,4) \bigg\}.
\end{align*}

To prove the above properties, we first estimate, using that $V,\fC_0\geq 0$, 
\[
\prod_{i\in\{1,3\}} \bigg|T^{d/2} V(W^{\ssup i}_T\!-\!W^{\ssup{i+1}}_T) - \fC_0\bigg| \leq  \prod_{i\in\{1,3\}} \bigg[T^{d/2} V(W^{\ssup i}_T\!-\!W^{\ssup{i+1}}_T)
+ \fC_0 \bigg] 
\]
and expand the last product. Then, we observe that for $\b$ small enough, H\"older's inequality and Lemma  \ref{lemma2.5-claim1} directly give:
\[
\begin{aligned}
 & \fC_0^2 \sup_{T>0} E_{0}^{\otimes 4} \bigg[ \mathrm e^{\b^2 \sum_{1\leq i < j \leq 4}  \int_0^T V(W^{\ssup i}_t\!-\!W^{\ssup j}_t)\dd t} \,\, \bigg]<\infty,\\
\text{and }\, & \fC_0^2 \lim_{T\to\infty}\,\, E_{0}^{\otimes 4} \bigg[ \mathrm e^{\b^2 \sum_{1\leq i < j \leq 4}  \int_0^T V(W^{\ssup i}_t\!-\!W^{\ssup j}_t)\dd t} \,\, \mathbf 1_A\bigg]=0.
\end{aligned}
\]
Moreover, switching from free Brownian motion to the Brownian bridge,
\begin{align*}
& E_{0}^{\otimes 4} \bigg[ e^{\b^2 \sum_{1\leq i < j \leq 4}  \int_0^T V(W^{(i)}_t\!-\!W^{(j)}_t)\dd t} \,\,
 \prod_{i=\{1,3\}} \bigg(T^{d/2} V(W^{(i)}_T\!-\!W^{(i+1)}_T)  \bigg) \,\,\mathbf 1_A\bigg]\\
& = \int_{(\mathbb R^d)^4} \dd \mathbf y\,\, \prod_{i\in\{1,3\}}\left(T^{d/2} V(y_i-y_{i+1}) \rho(T,y_i) \rho(T,y_{i+1})\right) \\
& \qquad\bigotimes_{i=1}^4 E_{0,0}^{T,y_i} \bigg[ \mathrm e^{\b^2 \sum_{1\leq i < j \leq 4}  \int_0^T V(W^{(i)}_t\!-\!W^{(j)}_t)\dd t} \,\,\mathbf 1_A\bigg],
\end{align*}
where the same equality also holds without the indicator $\mathbf{1}_A$. A similar, even simpler, decomposition holds for 
$$
\fC_0 \; E_{0}^{\otimes 4} \bigg[ e^{\b^2 \sum_{1\leq i < j \leq 4}  \int_0^T V(W^{(i)}_t\!-\!W^{(j)}_t)\dd t} \,\,
\bigg(T^{d/2} V(W^{(1)}_T\!-\!W^{(2)}_T)  \bigg) \,\,\mathbf 1_A\bigg]\;.
$$
Since $\rho(T,y_{i+1}-r)\leq C T^{-d/2}$ and $V$ is compactly supported, we  finally obtain \eqref{eq:removingProofEst} and \eqref{eq:removingProofEst2}
by H\"older's inequality, Lemma \ref{lemma2-claim1} and Lemma \ref{lemma3-claim1}.

\qed

\medskip

\subsection{Proof of Proposition \ref{prop-claim4}.}\label{sec:proof-claim4}

If we denote by $\mathcal F_{[0,T/2]}$ the $\sigma$-algebra generated by all four Brownian paths until time $T/2$, then, using Markov's property, 
\[
\begin{aligned}
\mathcal T_2= E_0^{\otimes 4}[H_m\,\, H_{T-m,T}] 
&=E_0^{\otimes 4}\left[E_0^{\otimes 4}\left(H_m\,\, H_{T-m,T}\middle| \big(W_{T/2}^{\ssup i}\big)_{i=1}^4\right)\right] \\
&=E_0^{\otimes 4}\left[E_0^{\otimes 4}\left\{E_0^{\otimes 4}\left(H_m\,\, H_{T-m,T}\middle| \mathcal F_{[0,T/2]}\right) \middle | \big(W_{T/2}^{(\ssup i}\big)_{i=1}^4\right\} \right] \\
%&=E_0^{\otimes 4}\left[E_0^{\otimes 4}\left\{H_m\,\,E_0^{\otimes 4}\left( H_{T-m,T}\middle| \mathcal F_{[0,T/2]} \middle | \big(W_{T/2}^{(i)}\big)_{i=1}^4\bigg)\right\} \right]\\
%&=E_0^{\otimes 4}\left[E_0^{\otimes 4}\left\{H_m\,\,E_0^{\otimes 4}\left( H_{T-m,T}\middle| \big(W_{T/2}^{(i)}\big)_{i=1}^4\right)\middle | \big(W_{T/2}^{(i)}\big)_{i=1}^4\right\} \right]\\
&=E_0^{\otimes 4}\left[E_0^{\otimes 4}\left\{H_m\middle| \big(W_{T/2}^{\ssup i}\big)_{i=1}^4\right\} E_0^{\otimes 4}\left\{H_{T-m,T}\middle| \big(W_{T/2}^{\ssup i}\big)_{i=1}^4\right\}\right].
%E_0^{\otimes 4}\left[\bigg[E_0^{\otimes 4}H_m\middle| \big(W_{T/2}^{(i)}\big)_{i=1}^4 \, \left[E_0^{\otimes 4}H_{T-m,T}\middle| (W_{T/2}^{(i)})_{i=1}^4\right]\right].
\end{aligned}
\]
We will prove that there exists a constant $C<\infty$, such that: 
\begin{equation*}
\begin{gathered}
\text{(i)} \sup_{T>0} E_0^{\otimes 4}\left\{H_m\middle| \big(W_{T/2}^{\ssup i}\big)_{i=1}^4\right\} \leq C \qquad \text{(ii)} \sup_{T>0} E_0^{\otimes 4}\left\{H_{T-m,T}\middle| \big(W_{T/2}^{\ssup i}\big)_{i=1}^4\right\} \leq C,\\
\text{(iii)} \  E_0^{\otimes 4}\left\{H_m\middle| \big(W_{T/2}^{\ssup i}\big)_{i=1}^4\right\} \cvlaw E_0^{\otimes 4}\left[H_\infty\right], \text{ as } T\to\infty.
\end{gathered}
\end{equation*}
where $H_\infty$ is defined as $H_m$ with the time interval $[0,m]$ replaced by $[0,\infty)$, recall \eqref{eq-H-m}. 

Let us first conclude the proof of Proposition \ref{prop-claim4} assuming the above three assertions. The difference of the two first terms in \eqref{eq:prop-claim4} writes:
\begin{align*}
&\mathcal T_2 - E_0^{\otimes 4}\left[ H_\infty \right] E_0^{\otimes 4}\left[H_{T-m,T}\right]
 \\
&= E_0^{\otimes 4}\left[\left(E_0^{\otimes 4}\left\{H_m\middle| \big(W_{T/2}^{\ssup i}\big)_{i=1}^4\right\} - E_0^{\otimes 4}\left[ H_\infty \right]\right)
 E_0^{\otimes 4}\left\{H_{T-m,T}\middle| \big(W_{T/2}^{\ssup i}\big)_{i=1}^4\right\}\right],
\end{align*}
which goes to $0$ as $T\to\infty$ by (i)-(iii), proving \eqref{eq:prop-claim4}. Finally, computations of Section 3.1 ensure that:
\[
\bigg[E_0^{\otimes 2}\bigg\{\mathrm e^{\beta^2 \int_{T-m}^T V(W^{\ssup 1}_s- W^{\ssup 2}_s) \dd s}\,\,\bigg(T^{d/2}V\left(W^{\ssup 1}_T- W^{\ssup 2}_T\right)- \fC_0\bigg)\bigg\}\bigg] \underset{T\to\infty}{\longrightarrow}0.
\]

We now owe the reader the proofs of (i)-(iii). To prove (i), we use H\"older's inequality to get 
\[
\begin{aligned}
E_0^{\otimes 4}\left\{H_m\middle| \big(W_{T/2}^{(i)}\big)_{i=1}^4\right\} &\leq \prod_{1\leq i < j \leq 4} E_0^{\otimes 4}\left[\mathrm e^{6\beta^2 \int_0^m V(W^{\ssup i}_t\!-\!W^{\ssup j}_t)\dd t}\middle| \big(W_{T/2}^{\ssup i}\big)_{i=1}^4\right]^{1/6} \\
&= \prod_{1\leq i < j \leq 4} E_{0,0}^{T/2,W_{T/2}^{\ssup i}-W_{T/2}^{\ssup j}} \bigg[\mathrm e^{6\beta^2 \int_0^m V(\sqrt 2 W_t)\dd t}\bigg]^{1/6} \\
&\leq \sup_{T,z} E_{0,0}^{T/2,z} \bigg[e^{6\beta^2 \int_0^{T/2} V(\sqrt 2 W_t)\dd t}\bigg] <\infty,
\end{aligned}
\]
by Lemma \ref{lemma2-claim1}. For (ii), we note that by Markov's property, 
\[
\begin{aligned}
E_0^{\otimes 4}\left\{H_{T-m,T}\middle| \big(W_{T/2}^{\ssup i}\big)_{i=1}^4\right\} &=\prod_{i\in\{1,3\}} E_{W_{T/2}^{\ssup i}-W_{T/2}^{\ssup{i+1}}} \left[ \mathrm e^{\beta^2 \int_{T/2-m}^{T/2} V(\sqrt 2 W_t) \dd t}\,\left(T^{d/2} V\left(\sqrt 2 W_{T/2}\right)- \fC_0\right)\right].
\end{aligned}
\]
We have:
\[
\fC_0\, E_{W_{T/2}^{\ssup i}-W_{T/2}^{\ssup{i+1}}} \bigg[ \mathrm e^{\beta^2 \int_{T/2-m}^{T/2} V(\sqrt 2 W_t) \dd t}\bigg]\leq \fC_0 \sup_{z} E_z \bigg[\mathrm e^{\beta^2 \int_0^\infty V(\sqrt 2 W_t) \dd t}\bigg]<\infty,
\]
while, for some constant $C'>0$, 
\[
\begin{aligned}
&E_{W_{T/2}^{\ssup i}-W_{T/2}^{\ssup{i+1}}} \bigg[ e^{\beta^2 \int_{T/2-m}^{T/2} V(\sqrt 2 W_t) \dd t}\,\, T^{d/2}V\left(\sqrt 2 W_{T/2}\right)\bigg] \\
&\leq C'\int_{\mathbb R^d} \dd z \,\,\, E_{0,W_{T/2}^{\ssup i}-W_{T/2}^{\ssup{i+1}}}^{T/2,z}\,\,\, \bigg[ \mathrm e^{\beta^2 \int_{T/2-m}^{T/2} V(\sqrt 2 W_t) \dd t}\bigg]\,\, V\left(\sqrt 2 z\right) \\
&\leq C'  \sup_{T,y,z}\,\,  E_{0,y}^{T/2,z} \bigg[ \mathrm e^{\beta^2 \int_0^{T/2} V(\sqrt 2 W_t)\,\,\dd t}\bigg] \,\, \int_{B(0,1)} \dd z \,\, V\left(\sqrt 2 z\right)\\
&<\infty,
\end{aligned}
\]
again by Lemma \ref{lemma2-claim1}. 

Finally, to prove (iii), we fix any bounded continuous  test function $f:\mathbb R\to\mathbb R$, so that 
\begin{align}
E_0^{\otimes 4} \bigg[ f\left(E_0^{\otimes 4}\left\{H_m\middle| \big(W_{T/2}^{\ssup i}\big)_{i=1}^4\right\}\right)\bigg]&= \int_{(\mathbb R^d)^4} \dd \mathbf y \,\, f\bigg(E_{0,0}^{T/2,\mathbf y}\left[H_m\right]\bigg) \prod_{i=1}^4 \rho(T/2,y_i) \nn\\
&= \int_{(\mathbb R^d)^4} \dd \mathbf z\,\,  f\bigg(E_{0,0}^{T/2,\sqrt T \mathbf z}\left[H_m\right]\bigg)  \prod_{i=1}^4 \rho(1/2,z_i). \label{eq-claim3-prop-claim4}
\end{align}
Now, letting $T\to\infty$, we get similarly to \eqref{eq:BBToBM} that $E_{0,0}^{T/2,\sqrt T \mathbf z}\left[H_m\right]\to E_0^{\otimes 4}\left[H_\infty\right]$. By dominated convergence, the RHS of \eqref{eq-claim3-prop-claim4} converges to $f\left(E_0^{\otimes 4}\left[H_\infty\right]\right)$, implying (iii).
\qed

\section{Proof of Theorem \ref{th:h} and Theorem \ref{th:CVagainstTestFun}.}\label{sec:proof:th:h}

We start with the proof of Theorem \ref{th:h} for which we need to show that $\{\mathscr H_\e(t,x)\}_{t>0,x\in \rd} \cvfidi \{\mathscr H(t,x)\}_{t>0,x\in \rd}$ as $\e\to 0$. 
For the reader's convenience, we split this proof into two tasks which will be split into the next two sections.

%The proof of Theorem \ref{th:h} now splits into two main tasks. Recall that, for any $t>0$ and $x\in \rd$, 
%$$
%\mathscr H_\e(t,x)= \e^{1-\frac d 2} [h_\e(t,x)- \mathfrak h(\xi^{\ssup{\e,t,x}})], \qquad \mathscr H(t,x)= \sqrt \fC_3 \int_t^\infty \rho(2\sigma, x-y) \, \xi(\sigma,y) \dd \sigma \dd y.
%$$

\subsection{Convergence of finite dimensional distributions of the spatially indexed process $\{\mathscr H_\e(t,x)\}_{x\in \rd}$.}\label{sec-f.d.m-space}%{\color {blue} for different $x$'s}.}
We first show that, 
\begin{proposition}\label{prop0:th:h}
For $\beta\in (0,\beta_0)$, any fixed $t>0$, $k\in \N$ and $x_1,\dots,x_k\in \rd$, the joint distributions of $\big(\mathscr H_\e(x_1,t), \dots, \mathscr H_\e(x_k,t)\big)$ converges to that 
of $\big(\mathscr H(x_1,t), \dots, \mathscr H(x_k,t)\big)$. 
\end{proposition}
%Having established this fact, we will then derive the tightness property of the spatially indexed random field $\{\mathscr H_\e(t,x)\}_{x\in\rd}$. 

For the rest of this section, we will write, for any $\sigma>0$, 
\begin{equation}\label{eq:L}
\begin{aligned}
&\mathscr L^{\ssup\sigma}_T(x)\stackrel{\mathrm{(def)}}{=} T^{d/2} \bigg(\frac {\dd}{\dd t}\big\langle \mathscr Z(0), \mathscr Z(r)\big\rangle\bigg)_{\sigma T} - \gamma^2 \rho(2\sigma,x) \mathscr Z_{\sigma T}(0) \mathscr Z_{\sigma T}(r)\\
&r\stackrel{\mathrm{(def)}}{=} \sqrt T x,
\end{aligned}
\end{equation}
where $\gamma^2=\gamma^2(\b)$ is defined in \eqref{def-sigma-beta}.

The key step for the proof of Proposition \ref{prop0:th:h} is  
\begin{proposition}\label{prop1:th:h}
Under the assumptions of Theorem \ref{th:h}, for $\beta\in (0,\beta_0)$, for any $x\in \R^d$ and $\sigma>0$, 
$$
\L\stackrel{L^2(\mathbb P)}{\longrightarrow}0 \qquad\mbox{as}\quad T\to\infty. 
$$
\end{proposition} 
Recall the square integrable martingale
\begin{align}
\dd \sZ_T(r)&= \b E_r  \bigg[ \Phi_T(W)\,\ \phi \star \xi (T,W_T)\bigg]\ dT \;,\nn\\ 
\dd \langle  \sZ(0), \sZ(r) \rangle_T &= \b^2 E_{0,r}^{\otimes 2} \bigg[ \Phi_T(W^{\ssup 1})\Phi_T(W^{\ssup 2}) V\big(W_T^{\ssup 1}-W_T^{\ssup 2}\big)  \bigg]  \dd T \nn \\ 
&=  \b^2 \sZ_T(0) \sZ_T(r)  \times E_{0,r,\b,T}^{\otimes 2}  \left[ V(W_T^{\ssup 1}-W_T^{\ssup 2})  \right]  \dd T \;, \label{eq:exprWithPolymerMeas}
\end{align}
so that 
\begin{equation}
\bigg(\frac {\dd}{\dd t}\big\langle \mathscr Z(0), \mathscr Z(r)\big\rangle\bigg)_{\sigma T}
= \beta^2 E_{0;r}^{\otimes 2}\bigg[\Phi_{\sigma T}(W^{\ssup 1}) \, \Phi_{\sigma T}(W^{\ssup 2}) \, V\big(W^{\ssup 1}_{\sigma T}- W^{\ssup 2}_{\sigma T}\big)\bigg], 
\end{equation}
where we remind the reader that $E^{\otimes 2}_{x;y}$ denotes expectation with respect to two independent Brownian motions $W^{\ssup 1}$ and $W^{\ssup 2}$ starting at $x\in \rd$ and $y\in \rd$ respectively.
The last expression, combined with \eqref{eq:L} then implies that 
 \begin{equation}\label{eq1:prop1:th:h}
 \begin{aligned}
 \E \big[\L\big] &= E_{0,r}^{\otimes 2}  \bigg[ 
\mathrm e^{\b^2\int_0^{\sigma T} V(W^{\ssup 1}_t\!-\!W^{\ssup 2}_t)\dd t} \bigg( T^{\frac{d}{2}}  V\big(W_{\sigma T}^{\ssup 1}\!-\! W_{\sigma T}^{\ssup 2}\big) - \gamma^2 \,\rho(2\sigma,x) \bigg)\bigg]% {\longrightarrow} 0 
\\ 
&=  E_{r}  \bigg[
\mathrm e^{\b^2\int_0^{\sigma T} V(  W_{2t})\dd t} \bigg( T^{\frac{d}{2}}  V( W_{2\sigma T}) - \gamma^2 \, \rho(2\sigma,x) \bigg)\bigg]\\
 &=
\int_{\rd}  V( y) E_{0,r}^{2\sigma T,y}  \bigg[  \mathrm e^{\b^2\int_0^{\sigma T} V(W_{2t})\dd t} \bigg] T^{\frac{d}{2}}  \rho(2\sigma T,y-r) \dd y
 - \gamma^2 \rho(2\sigma,x) \E[ \sZ_{\sigma T}(0) \sZ_{\sigma T}(r)] 
\end{aligned}
\end{equation}
where  we again remind the reader that $E^{t,y}_{0,x}$ denotes expectation with respect to a  Brownian bridge conditioned to  start at $x\in \rd$ and 
pinned at $y\in \rd$ at time $t$.  As before, we will first show that

\begin{proposition}\label{prop2:th:h}
Under the assumptions of Theorem \ref{th:h} (i.e., for $\beta\in (0,\beta_0)$), for any $x\in \R^d$ and $\sigma>0$, 
$$
\E\left[\L\right]{\longrightarrow}0 \qquad\mbox{as}\quad T\to\infty. 
$$
\end{proposition} 

The proof of Proposition \ref{prop2:th:h} splits into two main steps. Again, we fix $m=m(T)$ such that $m\to \infty$ and $m=o(T)$ as $T\to\infty$. 

\begin{lemma}\label{lemma1:prop2:th:h}
Under the assumptions imposed in Proposition \ref{prop2:th:h}, we have, for any $x\in \rd$, (recall $r=\sqrt T x$)
\[
\sup_{y\in \rd} \bigg| E_{0,r}^{2\sigma T,y}  \big[  \mathrm e^{\b^2\int_0^{\sigma T} V(W_{2t})\dd t} \big] - E_{0,r}^{2\sigma T,y}  \big[  \mathrm e^{\b^2\int_{[0,m]\cup[\sigma T -m,\sigma T]} V(W_{2t})\dd t} \big] \bigg| \stackrel{T\to\infty}{\longrightarrow} 0\;.
\]
\end{lemma} 
\begin{proof}The proof of Lemma \ref{lemma1:prop2:th:h} follows exactly the same line of arguments as in Proposition \ref{prop-claim1}. 
\end{proof}
The next result will then conclude the proof of Proposition \ref{prop2:th:h}. 
\begin{lemma}\label{lemma2:prop2:th:h}
Under the assumptions imposed in Proposition \ref{prop2:th:h}, we have, for any $x\in \rd$ and for $r=x\sqrt T$,
$$
\sup_{y\in \rd} \bigg| E_{0,r}^{2\sigma T,y}  \big[  \mathrm e^{\b^2\int_{[0,m]\cup[\sigma T -m,\sigma T]} V(W_{2t})\dd t} \big] - E_{r}  \big[  \mathrm e^{\b^2\int_0^\infty V(W_{2t})\dd t} \big] E_y \big[\mathrm e^{\b^2\int_0^\infty V(W_{2t} )\dd t}\big]\bigg| \stackrel{T\to\infty}{\longrightarrow} 0\;.
$$
\end{lemma}
\begin{proof}
Note that when $x\ne 0$, the integrals over $[0,m]$ vanish in the limit and the claim reduces to showing 
$$
\sup_y \bigg| E_{0,r}^{2\sigma T,y}\bigg[\mathrm e^{\beta^2\int_{\sigma T-\sigma m}^{\sigma T} V(W_{2t}) \dd t}\bigg]-  E_y\big[\mathrm e^{\beta^2\int_0^\infty V(W_{2t} )\dd t}\bigg]\bigg| \to 0,
$$
which is straightforward to check. 
%%We define,  
%%$$
%%\mathscr T_1^{\ssup\sigma}= \int_{\rd} \dd z \,\, T^{d/2} \frac{\rho(\sigma T, z\sqrt T- r)\rho(\sigma T, y- z\sqrt T)}{\rho(2\sigma T, y-r) }\,\, E_{0,r}^{\sigma T,\sqrt Tz}\bigg[\mathrm e^{\b^2\int_0^m V(W_{2t} \dd t}\bigg] \, E_{0,y}^{\sigma T,\sqrt T z}\bigg[\mathrm e^{\b^2 \int_0^m V(W_{2t} \dd t}\bigg],
%%$$
%%and claim that 
%%\begin{equation}\label{claim:th:h}
%%\sup_z \bigg| E_{0,r}^{\sigma T,\sqrt Tz}\big[\mathrm e^{\b^2\int_0^m V(W_{2t} \dd t}\big]- E_{0,y}^{\sigma T,\sqrt T z}\big[\mathrm e^{\b^2 \int_0^m V(W_{2t} \dd t}\big] \bigg|  \stackrel{T\to\infty}{\longrightarrow} 0.
%%\end{equation}
%%Again as in Proposition \ref{prop-claim2}, we can estimate the difference of the expectations in the last display as 
%$$
%\begin{aligned}
%&E_{0,r}\bigg[\mathrm e^{\b^2 \int_0^m W_{2t} \, \dd t} \, \bigg(\frac{\rho(\sigma T-m,\sqrt T z- W_{2m})}{\rho(\sigma T, \sqrt T z- x\sqrt T)} - 1 \bigg)\bigg] \\
%&= E_{0,r}\bigg[\mathrm e^{\b^2 \int_0^m W_{2t} \, \dd t} \, \bigg(\frac{\rho(\sigma -m/T,\sqrt T z- W_{2m}/\sqrt T)}{\rho(\sigma T, z- x)} - 1 \bigg)\bigg].
%\end{aligned}
%$$
%Since 
%$$
%\sup_z \bigg(\frac{\rho(\sigma -m/T,\sqrt T z- W_{2m}/\sqrt T)}{\rho(\sigma T, z- x)} - 1 \bigg)\to 0
%$$
% as $T\to\infty$, the requisite claim \eqref{claim:th:h} now follows. Then the dominated convergence theorem and smallness of $\beta$ concludes the proof of the lemma. 
\end{proof}

We will now show 

\begin{lemma}\label{lemma3:prop2:th:h}
Under the assumptions imposed in Proposition \ref{prop2:th:h}, we have, for any $x\in \rd$,
\begin{itemize}
\item As $T\to\infty$, \begin{equation}\label{item1}
 \E\big[\L^2\big] = \mathscr T_2^{\ssup\sigma}(x) + o(1),
\end{equation}
where 
\[
\mathscr T_2^{\ssup\sigma}(x)= \left(E_{0;r}^{\otimes 2} \otimes E_{0;r}^{\otimes 2}\right)\big[H_m H_{\sigma T-m,\sigma T}\big] ,
\]
with $H_m$ and $H_{\sigma T-m,\sigma T}$ defined in \eqref{eq-H-m}. 
\item As $T\to\infty$, 
\begin{equation}\label{item2}
\begin{aligned}
&\bigg| \mathscr T_2^{\ssup\sigma}(x)- \bigg[E_{0;0}^{\otimes 2}\bigg(\mathrm e^{\beta^2 \int_0^\infty \dd t \,\, V(W^{\ssup 1}_t- W^{\ssup 2}_t) }\bigg)\bigg]^2 \\
&\qquad \times \bigg[E_{0;r}^{\otimes 2}\bigg(\mathrm e^{\beta^2 \int_{\sigma T-m}^{\sigma T} V(W^{\ssup 1}_t- W^{\ssup 2}_t) \, \dd t} \, \big[T^{d/2} V\big(W^{\ssup 1}_{\sigma T}- W^{\ssup 2}_{\sigma T}\big) - \gamma^2 \rho(2\sigma,x)\big]\bigg)\bigg]^2\bigg| \to 0.
\end{aligned}
\end{equation} 
\end{itemize}
In particular, $\mathscr T_2^{\ssup\sigma}(x)\to 0$ and $\mathscr L_T^{(\sigma)}(x) \to 0$ in $L^2$.
\end{lemma}
\begin{proof}
Again, it suffices to consider the case $x\neq 0$. Note that 
\[
\begin{aligned}
&\E\big[\L^2\big]\\
&=\big(E_{0;r}^{\otimes 2} \otimes E_{0;r}^{\otimes 2}\big)\left[\prod_{j\in\{1,3\}} \left\{\mathrm e^{\beta^2 \int_0^T V(W^{\ssup i}_t-W^{\ssup{i+1}}_t)\dd t }\big[T^{d/2} V(W^{\ssup i}_T-W^{\ssup{i+1}}_T)-\gamma^2 \rho(2\sigma,x)\big]\right\}\right.\\
&\qquad\qquad\qquad\qquad\qquad\times \mathrm e^{\beta^2 \sum^\star \int_0^T \dd t \, V(W^{\ssup i}_t- W^{\ssup j}_t)}\bigg].
\end{aligned}
\]
Then a repetition of the exactly same arguments as in Proposition \ref{prop-claim3} prove \eqref{item1}. To deduce \eqref{item2},  shortening notations 
$ E_{0;r}^{\otimes 2} \otimes E_{0;r}^{\otimes 2}$ into $E_{0;r;0;r}^{\otimes 4}$, observe that 
$$
\mathscr T_2^{\ssup\sigma}(x)= E_{0;r;0;r}^{\otimes 4} \bigg[ E_{0;r;0;r}^{\otimes 4}\bigg\{ H_m \bigg|  \big(W^{\ssup i}_{\sigma T/2}\big)_{i=1}^4\bigg\}
E_{0;r;0;r}^{\otimes 4}
\bigg\{ H_{\sigma T-m} \bigg| \big(W^{\ssup i}_{\sigma T/2}\big)_{i=1}^4\bigg)\bigg\}\bigg]
$$
Hence, the left hand side in \eqref{item2} can be rewritten as 
$$
\begin{aligned}
&\mathscr T_2^{\ssup\sigma}(x)- \big(E_{0;0}\big[\mathrm e^{\beta^2\int_0^\infty V(W_{2t}) \dd t}\big]\big)^2 E_{0;r0;r}^{\otimes 4}\big[H_{\sigma T-m,\sigma T}\big] \\
&= E_{0;r;0;r}^{\otimes 4}\bigg[ \bigg\{E_{0;r;0;r}^{\otimes 4}  \bigg[H_m\bigg| \big(W^{\ssup i}_{\sigma T/2}\big)_{i=1}^4\bigg] - \bigg(E_{0;0}\bigg[\mathrm e^{\beta^2\int_0^\infty V(W_{2t}) \dd t}\bigg]\bigg)^2\bigg\}
\times E_{0;r;0;r}^{\otimes 4}\bigg\{ H_{\sigma T-m} \bigg| \big(W^{\ssup i}_{\sigma T/2}\big)_{i=1}^4\bigg\}\bigg] 
%\\ &\to 0.
\end{aligned}
$$
which tends to 0, as can be seen from the proof of Proposition \ref{prop-claim3}. 
\end{proof}

\begin{proof}[Proof of Proposition \ref{prop0:th:h}.]
%\noindent{\bf{Proof of Proposition \ref{prop0:th:h}.}} 
The proof follows a similar line of arguments as in Section \ref{subsec:secondProof}. In particular, it is enough to show that for all $t>0$,
\begin{equation*} \label{eq:CV_forSpaceVar}
 T^{(d-2)/4} \left(\log \sZ_\infty(\sqrt T x) - \log \sZ_{tT}(\sqrt T x)\right) \cvlaw \mathscr  H(t,x),
\end{equation*}
where the convergence holds jointly for finitely many $x$'s.
 Note that again by It\^o's formula, $\log \sZ_T(x)= N_T(x)- \frac 12 \langle N(x)\rangle_T$, where 
\[
N_T(x)=\beta \int_0^T \int_{\R^d} E_{x,\b,t} [ \phi (y-W_t) ] \xi(t,y) \dd y \dd t,
\]
is a martingale with cross-bracket
\begin{align*}
&\quad \langle N(x),N(y)\rangle_T = %\frac{\b^2}{2} 
\b^2
\int_0^T E_{x,y,\b,s}^{\otimes 2}  \left[ V(W_s^{\ssup 1}-W_s^{\ssup 2})  \right]  \dd s,
\end{align*}
where we remind the reader that $E_{x,\beta,t}$ denotes the expectation w.r.t. the polymer measure, while $E_{x,y,\beta,t}^{\otimes 2}$ denotes the same w.r.t. the product 
polymer measure (on the same environment) defined on two independent Brownian paths starting at $x$ and $y$. 

We now use the multidimensional version of  the functional central limit for martingales (\cite[Theorem 3.11]{JS87}) combined with Proposition \ref{prop1:th:h} to conclude that the
 sequence of rescaled martingales
\begin{equation} \label{eq:def_RescaledMartingale}
\left(N^{\ssup T}(x)\right):\tau\to T^{(d-2)/4}(N_{\tau T}(\sqrt{T}x)-N_{tT}(\sqrt{T}x)), \quad \tau\geq t,
\end{equation}
converge in the sense of finite dimensional distributions to a Gaussian field. For all $r=xT^{1/2}$ and $r'=yT^{1/2}$,
\begin{equation}
\begin{aligned}
\langle N^{\ssup T}(x), N^{\ssup T}(y)\rangle_\tau & = T^{(d-2)/2} \b^2 \int_{tT}^{\tau T} E_{r,r',\b,s}^{\otimes 2}  \left[ V(W_s^{\ssup 1}-W_s^{\ssup 2})  \right]  \dd s \nonumber \\
& = \int_{t}^{\tau}  T^{d/2} \b^2 E_{r,r',\b,\sigma T}^{\otimes 2}  \left[ V(W_s^{\ssup 1}-W_s^{\ssup 2})  \right]  \dd \sigma \nonumber 
%\\ &\to \fC_3 \int_{t}^{\tau}   \rho(2\sigma,y-x) \dd\sigma 
\end{aligned}
\end{equation}
and we claim that
\begin{equation}\label{eq4.8}
\lim_{T \to \8} \sup_{x,y} \bigg\| \langle N^{\ssup T}(x), N^{\ssup T}(y)\rangle_\tau -  \gamma^2 \int_{t}^{\tau}   \rho(2\sigma,y-x) \dd\sigma \bigg\|_1 =0\;.
\end{equation}
%
%
%
%as $T\to\infty$ in probability, by Proposition \ref{prop1:th:h}. \textcolor{red}{(I think some care has to be taken to justify convergence 
%\eqref{eq:CV_twoPointBracket}. In particular, it is not clear to me that we can use the same argument relying on the set $A_\e$ as in the proof of 
%Theorem \ref{th:cvGP}. The problem is that the starting points of the partition functions are now varying with $T$. Just below, I tried to give a proof 
%using negative moments, leading to a convergence in $L^1$-norm).}
Note that for $\tau=\infty$, the second term in the norm in \eqref{eq4.8} equals $\mathrm{Cov}(\mathscr H(t,x)\mathscr H(t,y))$, so that, assuming \eqref{eq4.8}, this concludes the proof of Proposition \ref{prop0:th:h} by the same arguments as exposed in Section \ref{subsec:secondProof}.

%\begin{proof}[\textcolor{red}{Proof of convergence \eqref{eq:CV_twoPointBracket} in $L^1$-norm}]
%For $r=xT^{1/2}$ and $r'=yT^{1/2}$,
%\begin{align*}
%\langle N^{\ssup T}(x), N^{\ssup T}(y)\rangle_\tau & = T^{(d-2)/2} \b^2 \int_{tT}^{\tau T} E_{r,r',\b,s}^{\otimes 2}  \left[ V(W_s^{\ssup 1}-W_s^{\ssup 2})  \right]  \dd s \\
%& = \int_{t}^{\tau}  T^{d/2} \b^2 E_{r,r',\b,\sigma T}^{\otimes 2}  \left[ V(W_s^{\ssup 1}-W_s^{\ssup 2})  \right]  \dd \sigma,
%\end{align*}
%so that, 

We now prove the uniform limit \eqref{eq4.8}.
By \eqref{eq:exprWithPolymerMeas} and Cauchy-Schwarz inequality,
\begin{align*}&\E\left|\langle N^{\ssup T}(x), N^{\ssup T}(y)\rangle_\tau - \gamma^2 \int_{t}^{\tau}   \rho(2\sigma,y-x) \dd\sigma  \right|\\
& =  \E\left|\int_t^\tau \frac{1}{\sZ_{\sigma T}(r)\sZ_{\sigma T}(r')} \left( T^{d/2} \left(\frac {\dd}{\dd t}\big\langle \mathscr Z(r), \mathscr Z(r')\big\rangle\right)_{\sigma T}- \gamma^2 \rho(2\sigma,y-x)\right)\dd \sigma \right| \\
&\leq \int_t^\tau \left\Vert\frac{1}{\sZ_{\sigma T}(r)\sZ_{\sigma T}(r')}\right\Vert_2 \left\Vert T^{d/2} \left(\frac {\dd}{\dd t}\big\langle \mathscr Z(r), \mathscr Z(r')\big\rangle\right)_{\sigma T}- \gamma^2 \rho(2\sigma,y-x)\right\Vert_2 \dd \sigma.
\end{align*}
Recall that $\sup_T \E[|\log \sZ_T|^p]<\infty$ for all $p<0$ (see \eqref{eq:negative}). Hence, convergence \eqref{eq4.8}
%\eqref{eq:CV_twoPointBracket} 
is obtained from Proposition \ref{prop1:th:h} and  dominated convergence applied to the last line of the above display. Note that domination can be justified by the following bound for all $\sigma \geq t$:
\begin{equation} \label{eq:L2bound2pointsDerivative}
\left\Vert T^{d/2} \left(\frac {\dd}{\dd t}\big\langle \mathscr Z(r), \mathscr Z(r')\big\rangle\right)_{\sigma T}\right\Vert_2 \leq C \sigma^{-d/2},
\end{equation}
where $C$ is some finite constant independent of $x$, $y$, $T$ and $\sigma$ (independence with respect to $x,y$ will be useful for the proof of Theorem \ref{th:CVagainstTestFun}). It can be directly obtained 
%from Proposition \ref{prop1:th:h} with $\sigma = 1$, 
by observing that
\begin{equation} \label{eq:UniformL2Bound}
\sup_{T>0} \sup_{x,y\in\mathbb{R}^d} \left\Vert T^{d/2} \left(\frac {\dd}{\dd t}\big\langle \mathscr Z(r), \mathscr Z(r')\big\rangle\right)_{T}\right\Vert_2 < \infty,
\end{equation}
which we get by computing the second moment:
\begin{align*}
&\E \left[T^{d} \left(\frac {\dd}{\dd t}\big\langle \mathscr Z(0), \mathscr Z(r)\big\rangle\right)_{T}^2\right] \\
& = \left(E_{0;r}^{\otimes 2} \otimes E_{0;r}^{\otimes 2}\right)\left[e^{\beta^2 \sum_{i<j\leq 4} \int_0^T \dd t \, V\big(W^{\ssup i}_t- W^{\ssup j}_t\big)}\prod_{j\in\{1,3\}} \mathrm T^{d/2} V\big(W^{\ssup i}_T-W^{\ssup{i+1}}_T\big)\right]\\
& = \int_{(\mathbb R^d)^4} \dd \mathbf y\,\, \prod_{i\in\{1,3\}}\left(T^{d/2} V(y_i-y_{i+1}) \rho(T,y_i) \rho(T,y_{i+1}-r)\right) \\
& \qquad\bigg[\bigotimes_{i\in\{1,3\}} \big(E_{0,0}^{T,y_i} \otimes E_{0,r}^{T,y_{i+1}}\big)\bigg] \bigg[ \mathrm e^{\b^2 \sum_{1\leq i < j \leq 4}  \int_0^T V\big(W^{(i)}_t\!-\!W^{(j)}_t\big)\dd t}\bigg]. 
\end{align*}
By Lemma \ref{lemma2-claim1} and H\"older's inequality, the expectation over the Brownian bridges in the last display is uniformly bounded for small enough $\b$. Since $\rho(T,y_{i+1}-r)\leq C T^{-d/2}$ and $V$ is compactly supported, we obtain \eqref{eq:UniformL2Bound}.
%
%\textcolor{red}{(Remark: since the RHS of \eqref{eq:L2bound2pointsDerivative} is integrable in on $[t,\infty)$, it seems that there is in fact convergence 
%in $L^1$ uniformly in $\tau>t$.)}
\end{proof}

%%%%%%%%%%
\subsection{Proof of Theorem \ref{th:h}: Convergence of space-time finite dimensional distributions of $\{\mathscr H_\e(t,x)\}_{t>0,x\in \rd}$. }\label{4.2}
Note that to derive Theorem \ref{th:h} it suffices to show that for $\beta\in (0,\beta_0)$ the following convergence of the joint distribution of a finite vector holds:
$$
(\mathscr H_\e(t,x), \mathscr H_\e(s,y),\dots,) \to (\mathscr H(t,x),\mathscr H(s,y), \dots).
$$
 The proof builds on the lines of arguments as that of Proposition \ref{prop0:th:h}. %To avoid repetition we will only sketch the argument quite briefly. 
Recall that with $\xi^{\ssup{\e,t}}(s,y)=\e^{(d+2)/2}\xi(t-\e^2 s, \e y)$ we have $u_\e(t,x)= \sZ_{t/\e^2}(\xi^{\ssup{\e,t}}; x/\e)$ and therefore, for all $s\leq t$,
\begin{equation} \label{eq:couplingProporty}
\big(u_\e(s,y), u_\e(t,x)\big)= \Bigg( \sZ_{\frac{t-s}{\e^2}, \frac t {\e^2}}\bigg(\frac y\e,\xi^{\ssup{\e,t}}\bigg), \sZ_{0,\frac t{\e^2}}\bigg(\frac x \e, \xi^{\ssup{\e,t}}\bigg)\Bigg)
\end{equation}
where for any $0\leq A \leq B$ and $X\in \rd$, we wrote
\begin{equation}\label{PhiAB}
\begin{aligned}
&\sZ_{A,B}(X;\xi)=  %  E_{W_A=X}\big[\Phi_{A,B}(W)\big],\qquad\mbox{with}\\
E\big[\Phi_{A,B}(W) \big \vert {W_A=X}\big],\qquad\mbox{with}\\
&\Phi_{A,B}(W)=\exp\bigg\{\beta\int_A^B\int_{\rd} \phi(W_s-y) \xi(s,y) \dd s \dd y- \frac{\beta^2}2 (B-A) V(0)\bigg\}.
\end{aligned}
\end{equation}
In this section we will also write 
\[
T=\e^{-2}, u=t-s\in[0,t) \qquad\mbox{and}\,\,\, \sigma>u.
\]
Starting from \eqref{eq:couplingProporty}, we can restrict ourselves to $y=0$ by shift invariance and compute as before: 
\begin{equation*}
\begin{aligned}
&\Bigg(\frac{\dd}{\dd t} \bigg\langle \sZ_{Tu,\cdot}(0;\cdot), \sZ_{0;\cdot}(\sqrt T x,\cdot)\bigg\rangle\Bigg)_{T\sigma}\\
&= \beta^2\big(E_{Tu,0}\otimes E_{0,\sqrt T x}\big)\big[\Phi_{Tu,T\sigma}(W^{\ssup 1})\Phi_{0,T\sigma}(W^{\ssup 2}) \, V\big(W^{\ssup 1}_{T\sigma}- W^{\ssup 2}_{T\sigma}\big)\big].
\end{aligned}
\end{equation*}
As in \eqref{eq:L} if we now write 
\begin{equation}\label{eq:L1}
\begin{aligned}
\mathscr L^{\sigma,u}_T(x)&= T^{d/2} \bigg[\beta^2\big(E_{Tu,0}\otimes E_{0,\sqrt T x}\big)\big[\Phi_{Tu,T\sigma}(W^{\ssup 1})\Phi_{0,T\sigma}(W^{\ssup 2}) \, V\big(W^{\ssup 1}_{T\sigma}- W^{\ssup 2}_{T\sigma}\big)\big]\bigg] \\
&\qquad\qquad\qquad- \gamma^2 \rho(2\sigma-u,x) \sZ_{Tu,T\sigma}(0) \sZ_{0,T\sigma}(\sqrt Tx),
\end{aligned}
\end{equation} 
then we again need to show 
\begin{proposition}\label{prop4}
Under the same assumption as in Theorem \ref{th:h}, for any $\sigma>u\geq 0$ and $x\in\rd$
\begin{itemize}
\item As $T\to\infty$,
\begin{equation}\label{prop4eq1}
\E\left[\mathscr L^{\sigma,u}_T(x)\right]\to 0.
\end{equation}
\item As $T\to\infty$,
\begin{equation}\label{prop4eq2}
\E\left[\mathscr L^{\sigma,u}_T(x)^2\right]\to 0.
\end{equation}
\end{itemize}
Therefore, $\mathscr L^{\sigma,u}_T(x)\stackrel{L^2(\mathbb P)}{\longrightarrow} 0$.
\end{proposition}
\begin{proof}We will first carry out 

\noindent{\bf{Proof of \eqref{prop4eq1}:}} With $\mathscr L^{\ssup{\sigma-u}}_T(\cdot)$ defined in \eqref{eq:L}, we will show that 
\begin{equation}\label{prop4eq11}
\E[\mathscr L^{\sigma,u}_T(x)] = \int \dd z \rho(u, z-x) \E[\mathscr L_T^{\ssup{\sigma-u}}(z)] + R_T ,
\end{equation}
where
\[R_T = - \gamma^2 \int \dd z \rho(u,z\!-\!x)\left(\rho(2\sigma\!-\!u,x) -\rho(2\sigma\!-\!2u,z)\right) \E\left[\sZ_{0,T(\sigma\!-\!u)}(0)\sZ_{0,T(\sigma\!-\!u)}(\sqrt T z)\right],
\]
 and that
\begin{equation}\label{prop4eq12}
\sup_{z\in \rd, T>0} \E[|\mathscr L^{(\sigma- u)}_T(z)|]< \infty.
\end{equation}
By dominated convergence, Proposition \ref{prop2:th:h} and \eqref{prop4eq12} justify that the first term of the RHS of \eqref{prop4eq11} vanishes as $T\to\infty$. Then, the covariance computation in \eqref{eq:covM0} implies that the expectation in the summand of $R_T$ goes to $1$ for all $z\neq 0$. By dominated convergence and the Chapman-Kolmogorov property, we therefore obtain that $R_T$ also vanishes as $T\to \infty$, which in turn implies \eqref{prop4eq1}. 

To derive \eqref{prop4eq12}, we appeal to \eqref{eq:L} and \eqref{eq1:prop1:th:h}, so that 
$$
\begin{aligned}
\E[|\mathscr L^{\sigma,u}_T(z)|] &\leq \int_{\rd} \dd y \, V(y) E_{0,z\sqrt T}^{2(\sigma-u)T,y}\bigg[\mathrm e^{\beta^2 \int_0^{(\sigma-u)T} V(W_{2t}) \, \dd t}\bigg] \, T^{d/2} \rho(2(\sigma-u)T,y-z\sqrt T) \\
 &\qquad\qquad +\gamma^2\rho(2\sigma,z) E[\mathscr Z_{(\sigma-u)T}^2] \\
 &\leq C \int_{\rd} \dd y V(y) + \gamma^2 (4\pi\sigma)^{-d/2} \, \sup_T \E[\sZ_{(\sigma-u)T}^2] <\infty.
 \end{aligned}
$$
For the second upper bound in the above display, we used that in the first summand  the first expectation is uniformly bounded thanks to Lemma \ref{lemma2-claim1}, while 
$T^{d/2} \rho((2(\sigma-u)T, y-z\sqrt T)\leq C^\prime$ and for the second summand we used $\sup_T \E[\sZ_{(\sigma-u)T}^2] <\infty$ for $\beta\in (0,\beta_0)$ and moreover $\rho(2\sigma,z) \leq (4\pi\sigma)^{-d/2}$. Hence, \eqref{prop4eq12} is justified.

To prove \eqref{prop4eq11}, we note that (recall the definition of $\Phi_{A,B}(\cdot)$ in \eqref{PhiAB}),
 $$
 \begin{aligned}
 \E[\mathscr L^{\sigma,u}_T(x)]&= T^{d/2}\beta^2\big(E_{Tu,0}\otimes E_{0,\sqrt T x}\big)\bigg[\E\left[\Phi_{Tu,T\sigma}(W^{\ssup 1})\Phi_{0,T\sigma}(W^{\ssup 2})\right] \, V\big(W^{\ssup 1}_{T\sigma}- W^{\ssup 2}_{T\sigma}\big)\bigg] \\
&\qquad\qquad\qquad- \gamma^2 \rho(2\sigma-u,x) \E\left[\sZ_{Tu,T\sigma}(0) \sZ_{0,T\sigma}(\sqrt Tx)\right] \\
&=T^{d/2}\beta^2\big(E_{Tu,0}\otimes E_{0,\sqrt T x}\big)\bigg[\mathrm e^{\beta^2 \int_{Tu}^{T\sigma} V\big(W^{\ssup 1}_s-W^{\ssup 2}_s\big) \dd s}\,\, V\big(W^{\ssup 1}_{T\sigma}- W^{\ssup 2}_{T\sigma}\big)\bigg] \\
&\qquad\qquad\qquad- \gamma^2 \rho(2\sigma-u,x) \int \dd z \rho(u,z-x) \E\left[\sZ_{0,T(\sigma-u)}(0) \sZ_{0,T(\sigma-u)}(\sqrt Tz)\right] \end{aligned}
$$
where the second identity follows from covariance structure of the white noise, the Markov property \eqref{eq:markov} and the diffusive scaling of the heat kernel.
Thus, \eqref{prop4eq11} follows from the above display and the following computation:
\begin{align*}
&\big(E_{Tu,0}\otimes E_{0,\sqrt T x}\big)\bigg[\mathrm e^{\beta^2 \int_{Tu}^{T\sigma} V\big(W^{\ssup 1}_s-W^{\ssup 2}_s\big) \dd s}\,\, V\big(W^{\ssup 1}_{T\sigma}- W^{\ssup 2}_{T\sigma}\big)\bigg]\\
& = \int \rho(u,z-x) \left(E_{0,0}\otimes E_{0,\sqrt T z}\right) \left[\mathrm e^{\beta^2 \int_{0}^{T(\sigma-u)} V\big(W^{\ssup 1}_s-W^{\ssup 2}_s\big) \dd s}\, V\left(W^{\ssup 1}_{T(\sigma-u)}- W^{\ssup 2}_{T(\sigma-u)}\right)  \right].
\end{align*}

We will now provide the proof of 

\noindent{\bf{Proof of \eqref{prop4eq2}:}} %We will again follow the philosophy of proving \eqref{prop4eq1} and show that 
%\begin{equation}\label{prop4eq2.5}
%\E[\mathscr L^{\sigma,u}_T(x)^2] - \int \int \dd z_1 \, \dd z_3 \,\, \rho(u,z_1)\rho(u,z_3) \E[\mathscr L^{\ssigma,u}_T(z_1) \, \mathscr L_T^{\ssup{\sigma,u}}_T(z_3)] \to 0.
%\end{equation}
Note that 
\begin{equation}\label{prop4eq3}
\begin{aligned}
&\E[\mathscr L^{\sigma,u}_T(x)^2]
= \left(E^{\otimes 2}_{(Tu,0);(0,\sqrt T x)} \otimes E^{\otimes 2}_{(Tu,0);(0,\sqrt T x)}\right)\bigg[\prod_{i\in\{1,3\}}\bigg\{\mathrm e^{\beta^2\int_{Tu}^{T\sigma} V(W^{\ssup i}_s-W^{\ssup{i+1}}_s)\dd s}\\
&\qquad\qquad\qquad\qquad\qquad\qquad\qquad\qquad\times \big[T^{d/2}V\big(W^{\ssup i}_{T\sigma}-W^{\ssup{i+1}}_{T\sigma}\big) 
- \rho(2\sigma-u,x)\gamma^2\big]\bigg\} \\
&\qquad \qquad\qquad\qquad\qquad\times \mathrm{exp}\bigg\{\beta^2 \sum_{\heap{i<j}{(i,j)\ne (2,4), (1,3)}} \int_{Tu}^{T\sigma} V(W^{\ssup i}_s- W^{\ssup{i+1}}_s) \dd s\bigg\} \,\, \mathrm e^{\beta^2\int_0^{T\sigma} V(W^{\ssup 2}_s- W^{\ssup 4}_s)\, \dd s}\bigg]
\end{aligned}
\end{equation}
Applying the same arguments for the proof of Proposition \ref{prop-claim3} once more, 
we now derive \eqref{prop4eq2}. %We refrain from spelling out the details. 
\end{proof}

%%Also  for any $z_1, z_3\in \rd$ and with $\mathscr L^{\sigma-u}_T(\cdot)$ defined in \eqref{eq:L}, we have 
%\begin{equation}\label{prop4eq4}
%\begin{aligned}
%&\E\big[\mathscr L^{\sigma-u}_T(z_1) \mathscr L_T^{\sigma-u}(z_3)\big]\\
%&= \big[E^{\otimes 2}_{0;\sqrt T z_1}\otimes E^{\otimes 2}_{0;\sqrt T z_3}\big]\bigg[\prod_{i\in\{1,3\}}\bigg\{\mathrm e^{\beta^2\int_0^{T(\sigma-u)} V(W^{\ssup i}_s-W^{\ssup{i+1}}_s)\dd s}\,\big[T^{d/2}V\big(W^{\ssup i}_{T(\sigma-u)}-W^{\ssup{i+1}}_{T(\sigma-u)}\big) 
%- \rho(2(\sigma-u),z_i)\fC_3\big]\bigg\} \\
%&\qquad \times \mathrm{exp}\bigg\{\beta^2 \sum_{\heap{i<j}{(i,j)\ne (2,4)}} \int_{Tu}^{T\sigma} V(W^{\ssup i}_s- W^{\ssup{i+1}}_s) \dd s\bigg\} \,\, \mathrm e^{\beta^2\int_0^{T\sigma} V(W^{\ssup 2}_s- W^{\ssup 4}_s)\, \dd s}\bigg]
%\end{aligned}
%$$
%Now the proof of 

\noindent{\bf{Proof of Theorem \ref{th:h}:}} 
 To derive Theorem \ref{th:h}, %it suffices to follow a similar line of arguments as the proof of Proposition \ref{prop0:th:h}.
 by property \eqref{eq:couplingProporty} for general vector length, it is enough to show that for all $t>0$ and all $u_1,\dots,u_n\in [0,t)$, $x_1,\dots,x_n\in\mathbb{R}^d$,
\begin{align}
&\left(\log \sZ_{T u_1,\infty}(\sqrt T x_1)-\log \sZ_{ T u_1, T t}(\sqrt T x_1)\,,\dots,\log \sZ_{T u_n,\infty}(\sqrt T x_n)-\log \sZ_{ T u_n, T t}(\sqrt T x_n)\right)\nn\\
&\qquad  \cvlaw \left(\mathscr H(t-u_1,x_1),\dots, \mathscr H(t-u_n,x_n)\right),\quad \text{as } T\to\infty.\label{eq:JointCVlogZ}
\end{align}
Again, we define 
\begin{align*}
N_{U,T}(X) = \b \int_U^{T} \int_{\mathbb{R}^d} E_{X,U,\sigma}\left[ \phi(y-W_\sigma)\right] \xi(\sigma,y)\dd \sigma \dd y,
\end{align*}
where $E_{X,A,B}$ denotes the expectation with respect to the polymer measure 
$$
P_{X,A,B}(\cdot)= P_{X,\beta,A,B}(\cdot)= \frac 1 {\sZ_{A,B}(X)} \, E_{W_U=X}\big[\Phi_{A,B}(W)\,\mathbf 1_\cdot\big]
$$
conditional on the Brownian path to start at $X\in \rd$ at time $U$, and $\Phi_{A,B}(\cdot)$ is defined in \eqref{PhiAB}. Then,
\begin{equation} \label{eq:decompIto}
\log \sZ_{U,T}(x) = N_{U,T}(x) - \frac{1}{2}\langle N_{U,\cdot}(x) \rangle_T,
\end{equation}
where the bracket satisfies
\begin{equation} \label{eq:exprBracketN2times}
\frac{\dd}{\dd t}\langle N_{U,\cdot}(Y),N_{0,\cdot}(X)\rangle_{T}= E_{U,Y,T}\otimes E_{0,X,T} \left[ V\left(W^{(1)}_{T}- W^{(2)}_{T} \right)\right],
\end{equation}
%so that we directly get from \eqref{prop4eq2} that:
%\begin{equation*}
%T^{d/2}\left(\frac{\dd}{\dd t} \langle N_{Tu,\cdot}(\sqrt{T}y),N_{0,\cdot}(\sqrt{T}x)\rangle\right)_{T\sigma} \overset{\IP}{\longrightarrow} \gamma^2 \rho(2\sigma - u,y-x).
%\end{equation*}
so that the following result holds:
\begin{lemma} \label{prop:twopointscorr}
Let $t\geq 0$ and define, for $u\leq t \leq \tau$, the martingales
\[(N^{\ssup{T,t}}(u,x)):\tau\to T^{(d-2)/4}\left(N_{Tu,{T\tau}}(\sqrt{T} x)-N_{Tu,Tt}(\sqrt{T}x)\right).\] Then, for all $u_1,\dots,u_n\in[0,t)$, $x_1,\dots,x_n\in \rd$,
\begin{itemize}
\item  As $T\to\infty$,
\begin{equation}\label{it1}
\langle N^{\ssup{T,t}}(u_1,x_1), N^{\ssup{T,t}}(u_2,x_2) \rangle_\tau \overset{L^1}{\longrightarrow} \gamma^2 \int_t^\tau \rho(2\sigma - (u_1+u_2),x_1-x_2) \dd \sigma.
\end{equation}
\item Moreover,
\begin{equation}\label{it2}
 \mathrm{Cov}\big(\mathscr H(t-u_1,x_1), \mathscr H(t-u_2,x_2)\big)= \gamma^2\int_t^\infty \rho(2\sigma -(u_1+u_2),x_1-x_2) \dd \sigma.
\end{equation}
%as $T\to\infty$ and as a space processes in the $x$-variable, 
%\begin{align*}
%&\left(T^{-\frac{d-2}4}\left(\log \sZ_{Tu,\infty}(\sqrt{T}x) - \log \sZ_{Tu,Tt}(\sqrt{T}x)\right),T^{-\frac{d-2}4}\left(\log \sZ_{0,\infty}(\sqrt{T}x) - \log \sZ_{0,Tt}(\sqrt{T}x)\right)\right)\\
%Hence, as $\e \to 0$,
%\begin{equation}\label{it3}
%\big(\mathscr H_\e(t-u_1,x_1),\dots, \mathscr H_\e(t-u_n,x_n)\big)
%\cvlaw (\mathscr H(t-u_1,x_1),\dots, \mathscr H(t-u_n,x_n).
%\end{equation}
\end{itemize}
\end{lemma}
\begin{proof}
First observe that for $u_2\leq u_1$,
\[\left(N^{\ssup{T,t}}_\tau(u_1,x_1), N^{\ssup{T,t}}_\tau(u_2,x_2)\right) \eqlaw \left(N^{\ssup{T,t-u_2}}_{\tau-u_2}(u_1-u_2,x_1), N^{\ssup{T,t-u_2}}_{\tau-u_2}(0,x_2)\right),\]
so that it suffices to prove convergence \eqref{it1} when $u_2=0$, and this is obtained from \eqref{eq:exprBracketN2times}, \eqref{eq:L1} and \eqref{prop4eq2} in the same way as \eqref{eq4.8} was proved.  \eqref{it2} is obtained by a simple covariance computation.
\end{proof}
Now, convergence \eqref{eq:JointCVlogZ} follows as in the proof from Section \ref{subsec:secondProof}, that is by appealing to the multi-dimensional CLT for martingales which will give that for all $t>0$, the martingales $N^{(T,t)}(u,x)$ will converge jointly in $u$ and $x$ to a Gaussian processes with covariance structure given by \eqref{it1}, then neglecting the bracket part in \eqref{eq:decompIto} in the scaling limit and finally letting $\tau\to\infty$ (observe that the RHS of \eqref{it2} equals the RHS of \eqref{it1} when $\tau=\infty$).

%Now, the proof is concluded by \eqref{it3} and the observation that, for $u_1,\dots,u_n\in[0,t]$, $x_1,\dots,x_n\in \rd$ and $T=\e^{-2}$, 
%\begin{equation}
%\eqlaw \big(N^{\ssup{T,t}}(u_1,x_1),\dots, N^{\ssup{T,t}}(u_n,x_n)\big).
%\end{equation}

\subsection{Proof of Theorem \ref{th:CVagainstTestFun}.}\label{subsec:CVagainstTest}
Recall that we need to show that for all $t>0$,
\begin{equation}
\int_{\mathbb{R}^d} \mathscr H_\e(t,x) \varphi(x) \rmd x \cvlaw \int_{\mathbb{R}^d} \mathscr H(t,x) \varphi(x) \rmd x\;,
\end{equation}
as $T\to \infty$, where the convergence holds jointly for finitely many test functions $\varphi\in\mathcal{C}_c^\infty$.
Once again, it suffices to prove that
\begin{equation} \label{eq:CV_againstTest}
\int_{\mathbb{R}^d}  T^{(d-2)/4} \left(\log \sZ_\infty(\sqrt T x) - \log \sZ_{tT}(\sqrt T x)\right) \varphi(x) \rmd x \cvlaw \int_{\mathbb{R}^d} \mathscr H(t,x) \varphi(x) \rmd x,
\end{equation}
holds jointly for finitely many $\varphi$'s.
As in the proof of Proposition \ref{prop0:th:h}, we decompose $\log \sZ_T(x)$ in $N_T(x)- \frac 12 \langle N(x)\rangle_T$, so that convergence \eqref{eq:CV_againstTest} reduces to studying the joint convergence as $T\to\infty$ of the following family of martingales:
\[
(N^{\ssup T}(\varphi)):\tau\to \int_{\mathbb{R}^d} N^{\ssup T}_\tau(x) \varphi(x) \dd x,
\]
where $N^{\ssup T}(x)$ is defined by \eqref{eq:def_RescaledMartingale}.
%\[
%N_T(f)=\int_{\R^d} N_T(x) f(x) \dd x,
%\]
For all test functions $\varphi_1$ and $\varphi_2$ in $\mathcal{C}_c^\infty$, we compute the cross-bracket and find that
\begin{align*}
\langle N^{\ssup T}(\varphi_1),N^{\ssup T}(\varphi_2)\rangle_\tau & =
\textcolor{black}{\int_{\mathbb{R}^d}}\int_{\mathbb{R}^d} \langle N^{\ssup T}(x) , N^{\ssup T}(y) \rangle_\tau \, \varphi_1(x)\varphi_2(y) \dd x \dd y\\
& \to \textcolor{black}{\int_{\mathbb{R}^d}\int_{\mathbb{R}^d}} \left(\gamma^2 \int_{t}^{\tau}   \rho(2\sigma,y-x)  \dd\sigma \right) \varphi_1(x)\varphi_2(y)\textcolor{black}{\dd x \dd y}.
\end{align*}
where the convergence is in $L^1$-norm as $T\to\infty$ and comes from uniform convergence \eqref{eq4.8}.
The proof is again concluded by the multidimensional functional central limit for martingales (\cite[Theorem 3.11]{JS87}) and the observation that the last integral in the above display, for $\tau=\infty$, is the covariance function of the joint Gaussian variables $\int_{\mathbb{R}^d} \mathscr H(t,x) \varphi_1(x) \rmd x$ and $\int_{\mathbb{R}^d} \mathscr H(t,x) \varphi_2(x) \rmd x$.

\section{Proof of Theorem \ref{thm:stationary}.}\label{proof:stationary}
Recall that we need to show that, for any $\varphi\in \mathcal C^\infty_c(\rd)$, 
$\int_{\rd} \varphi(x) \e^{1-d/2} [  \hh(t,x)\! -\! \E[  \hh(t,x)] \;\dd x \cvlaw \int_{\rd} \dd x \varphi(x) \HH(t,x)$
as $\e\to 0$, with $\HH$ being the stationary solution of the additive noise equation \eqref{StatEW} with GFF marginal distribution $\mathscr H(0,x)$ defined in Theorem \ref{th:h}. 
With $\mathscr H_\e(t,x)=\e^{1-d/2}[h_\e(t,x)- \hh(t,x)]$ as in \eqref{eq:Heps} and $\overline{\mathscr H}_\e(t,x): = \e^{1-d/2}[h_\e(t,x)- \E[h_\e(t,x)] $
we can rewrite 
\begin{equation}\label{eq2:stationary}
\begin{aligned}
\int_{\rd} \varphi(x) \e^{1-d/2} [  \hh(t,x)\! -\! \E[  \hh (t,x)] \;\dd x & =
\int_{\rd} \varphi(x)   \mathscr H_\e(t,x) \;\dd x + \int_{\rd} \varphi(x)   \overline{\mathscr H}_\e(t,x) \;\dd x \\
&\qquad- \int_{\rd} \varphi(x) \E[ \mathscr H_\e(t,x)] \;\dd x. 
%\qquad +  
%\int_{\rd} \varphi(x) \e^{1-d/2} [  h_\e(t,x) - \E[  h_\e(t,x)] \;\dd x\;,
\end{aligned}
\end{equation}
By Theorem \ref{th:h} and Theorem \ref{th:CVagainstTestFun}, combined with uniform integrability stemming from Proposition \ref{prop:LpBoundLogZ}, we have $\lim_{\e \to 0}\E[  \mathscr H_\e] = 0$ and the third term on the (r.h.s) above vanishes. Then Theorem \ref{th:CVagainstTestFun} and \eqref{EW} imply that the (l.h.s) above converges in law to $\int_{\rd} \varphi(x)  [ \mathscr H(t,x) + \overline  {\mathscr H}(t,x) ] \; \dd x$ with $ \mathscr H$ solving the deterministic heat equation 
\eqref{eq:HE} with the GFF initial condition $\mathscr H(0,x)$
 and $\overline  {\mathscr H}$ solving the additive-noise equation \eqref{EW} with flat initial condition, provided that we have joint convergence of the first and second integral on the right hand side in the above display %$\e^{1-\frac d2} \int_{\rd} \varphi(x) [h_\e(t,x)- \hh(t,x)] \;\dd x$ and $\int_{\rd} \varphi(x) \e^{1-d/2} [  h_\e(t,x) - \E[  h_\e(t,x)] \;\dd x$ 
 with independent limits $\int_{\rd} \dd x \varphi(x) \mathscr H(t,x)$ and $\int_{\rd} \dd x \varphi(x) \overline{\mathscr H}(t,x)$, respectively. 
 The independent sum $\mathscr H(t,x) + \overline  {\mathscr H}(t,x)$ is further easily shown to be a stationary solution of the additive-noise equation in \eqref{StatEW} with GFF initial condition $\mathscr H(0,x)$.
 
 Thus, it remains to show the joint convergence of $\int_{\rd} \varphi(x)   \mathscr H_\e(t,x) \;\dd x$ and $\int_{\rd} \varphi(x)   \overline{\mathscr H}_\e(t,x) \;\dd x$ 
 with independent limits in \eqref{eq2:stationary}, for which we proceed as follows. 
We set $T=\e^{-2}t$ and assume w.l.o.g. that $t=1$. By \eqref{hZ} and \eqref{def:stationary} and any $u, v\in \mathbb R$, we have the equality (recall the definition 
of the space-time shift $\theta_{t,x}$ from \eqref{eq:markov}) %({\color{blue} On the third line of the display below, we want to use that $|\mathrm e^{\mathbf i u T^{(d-2)/4} \int \varphi}|=1$. Thus, don't we need to take absolute value of the expectation below? If yes, this would be fine for the result, but we should still write it.})
\begin{align}
&\mathbb E \bigg[\exp\bigg\{\mathbf iu \int \varphi (x)   {\mathscr H}_\e(t,x) dx + \mathbf iv \int \varphi (x) \overline  {\mathscr H}_\e(t,x) \dd x \bigg\} \bigg] \label{eq3:stationary}\\
&= \mathbb E \bigg[\exp\bigg\{\mathbf i u \int \varphi (x)  T^{\frac{d-2}4}  \log \frac{\mathscr Z_T(x \sqrt T)}{ \mathscr Z_\8(x \sqrt T)} +
\mathbf  i v \int \varphi (x) T^{\frac{d-2}4}  \big[\log \mathscr Z_T(x \sqrt T)- \mathbb E \log \mathscr Z_T(x \sqrt T)\big]
dx \bigg\} \bigg] \nonumber\\
&= \mathbb E \bigg[ \exp\bigg\{-\mathbf iu \int \varphi (x)  T^{\frac{d-2}4} \bigg[\bigg( \log \frac{\mathscr Z_\8(x \sqrt T)}{ \mathscr Z_T(x \sqrt T)} -\frac{\mathscr Z_\8(x \sqrt T)}{ \mathscr Z_T(x \sqrt T)}  +1\bigg) +  \bigg( \frac{\mathscr Z_\8(x \sqrt T)}{ \mathscr Z_T(x \sqrt T)}  - E_{x\sqrt T} \big( \mathscr Z_\8 \circ \theta_{T,W(T)}  \big) \bigg)\bigg] \dd x \bigg\}  \nonumber
\\
&\qquad\qquad \times %\exp\bigg\{-\mathbf iu \int \varphi (x)  T^{\frac{d-2}4} \bigg[ \frac{\mathscr Z_\8(x \sqrt T)}{ \mathscr Z_T(x \sqrt T)}  - E_{x\sqrt T} \big( \mathscr Z_\8 \circ \theta_{T,W(T)}  \big) \bigg] \dd x\bigg\} 
%\times 
\exp\bigg\{-\mathbf iu \int \varphi(x) T^{\frac{d-2}4} E_{x\sqrt T} \big( \mathscr Z_\8 \circ \theta_{T,W(T)}  - 1 \big) \dd x\bigg\} \nonumber
\\
&\qquad\qquad  \times \exp\bigg\{
 \mathbf i v \int \varphi (x) T^{\frac{d-2}4}  \big[\log \mathscr Z_T(x \sqrt T)- \mathbb E \log \mathscr Z_T(x \sqrt T)\big] \dd x \bigg\} \bigg] \nonumber
 \\   
&= \mathbb E \bigg[  \exp\bigg\{-\mathbf i u \int \varphi(x) T^{\frac{d-2}4} E_{x \sqrt T} \big( \mathscr Z_\8 \circ \theta_{T,W(T)} -1  \big) \dd x %+ \mathbf i v \int T^{\frac{d-2}4}  \varphi(x) \big[\log \mathscr Z_T(x \sqrt T)- \mathbb E \log \mathscr Z_T(x \sqrt T)\big]  \dd x
\bigg\} \bigg] \nonumber \\
&\qquad\qquad  \times \exp\bigg\{ \mathbf i v \int T^{\frac{d-2}4}  \varphi(x) \big[\log \mathscr Z_T(x \sqrt T)- \mathbb E \log \mathscr Z_T(x \sqrt T)\big]  \dd x \bigg\} \bigg ] 
+ o(1) \nonumber 
 \\   
&= \mathbb E \bigg[\exp\bigg\{-\mathbf i u \int \varphi(x) T^{\frac{d-2}4} E_{x \sqrt T} \big( \mathscr Z_\8 \circ \theta_{T,W(T)} -1  \big) \dd x \bigg\}    \bigg]  \nonumber  \\
&\qquad\qquad \times  \mathbb E \bigg[  \exp\bigg\{ \mathbf i v \int \varphi (x) T^{\frac{d-2}4}  
\big[\log \mathscr Z_T(x \sqrt T)- \mathbb E \log \mathscr Z_T(x \sqrt T)\big] \dd x \bigg\} \bigg] + o(1) \label{eq4:stationary}
\end{align}
with
 $o(1) \to 0$ as $T \to \8$, once we prove that 
\begin{eqnarray} \label{eq:f17f1}
\int \varphi (x)  T^{\frac{d-2}4} \bigg( \log \frac{\mathscr Z_\8(x \sqrt T)}{ \mathscr Z_T(x \sqrt T)} -\frac{\mathscr Z_\8(x \sqrt T)}{ \mathscr Z_T(x \sqrt T)}  +1\bigg)  \dd x\to 0,\quad\mbox{and}
\\ \label{eq:f17f2}
 \int \varphi (x)  T^{\frac{d-2}4} \bigg[ \frac{\mathscr Z_\8(x \sqrt T)}{ \mathscr Z_T(x \sqrt T)}  - E_{x \sqrt T} \big( \mathscr Z_\8 \circ \theta_{T,W(T)}  \big) \bigg] \dd x  \to 0
 \end{eqnarray}
 in probability. Then in the equality of \eqref{eq3:stationary} and \eqref{eq4:stationary} above, we can choose $v=0, u\in \R$ and $u=0,v\in \R$ which  
 imply the desired asymptotic factorization as $\e\downarrow 0$ of the two expectations in \eqref{eq3:stationary}. 
 It thus remains to prove \eqref{eq:f17f1}-\eqref{eq:f17f2}, whose validity we now show in the $L^1(\mathbb P)$-norm.
To prove \eqref{eq:f17f1}  with $\varphi$ integrable, we can use that
\begin{eqnarray} \label{eq:f17f4}
\mathbb E \bigg |\int \varphi (x)  T^{\frac{d-2}4} \bigg( \log \frac{\mathscr Z_\8(x \sqrt T)}{ \mathscr Z_T(x \sqrt T)} -\frac{\mathscr Z_\8(x \sqrt T)}{ \mathscr Z_T(x \sqrt T)}  +1\bigg)  \dd x \bigg|\leq
a_T(0) \int_{\rd} |\varphi(x) | \dd x 
\end{eqnarray}
where
\begin{equation} \label{eq:f17f3} 
a_T(x):=
\mathbb E \bigg |T^{\frac{d-2}4}  \bigg( \log \frac{\mathscr Z_\8(x)}{ \mathscr Z_T(x)} -\frac{\mathscr Z_\8(x)}{ \mathscr Z_T(x)}  +1\bigg ) \bigg |=a_T(0)
%=\|T^{\frac{d-2}4}  \left( \ln \frac{\mathscr Z_\8(0)}{ \mathscr Z_T(0)} -\frac{\mathscr Z_\8(0)}{ \mathscr Z_T(0)}  +1\right) \|_1
\end{equation}
Using  Theorem \ref{th:CVmarginalZ} together with \eqref{eq:lpbound} and finiteness of all negative moments of $\mathscr Z_\8$ and 
since $\mathscr Z_T$ has smaller negative moments, we see that $a_T(0) \to 0$ as $T \to \8$ (the convergence holds in $L^p(\mathbb P)$ for all $p <2$) which proves \eqref{eq:f17f1}.  We follow a similar argument for proving \eqref{eq:f17f2} by estimating 
\begin{equation}\nn
\mathbb E \bigg |  \int \varphi (x)  T^{\frac{d-2}4} \bigg[ \frac{\mathscr Z_\8(x \sqrt T)}{ \mathscr Z_T(x \sqrt T)}  - E_{x \sqrt T} \big( \mathscr Z_\8 \circ \theta_{T,W(T)}  \big) \bigg] \dd x \bigg |
\leq b_T(0) \times \int |\varphi| \dd x
\end{equation}
where
\begin{equation} \nn
b_T(x):= \mathbb E \bigg | T^{\frac{d-2}4} \bigg[\frac{\mathscr Z_\8(x \sqrt T)}{ \mathscr Z_T(x \sqrt T)}  - E_{x \sqrt T} \big( \mathscr Z_\8 \circ \theta_{T,W(T)}  \big) \bigg] \bigg | \;.
%\mathbb E |T^{\frac{d-2}4}  \left( \ln \frac{\mathscr Z_\8(x)}{ \mathscr Z_T(x)} -\frac{\mathscr Z_\8(x)}{ \mathscr Z_T(x)}  +1\right) |=a_T(0)
%=\|T^{\frac{d-2}4}  \left( \ln \frac{\mathscr Z_\8(0)}{ \mathscr Z_T(0)} -\frac{\mathscr Z_\8(0)}{ \mathscr Z_T(0)}  +1\right) \|_1
\end{equation}
As above,  we show that $b_T(0) \to 0$ as $T \to \8$, by bounding the negative moments of $\mathscr Z_T$, and proving that the following norm vanishes :
(recall \eqref{eq:markov})
\begin{eqnarray} \label{eq:f17f3} 
%\phantom{}&
c_T&=& T^{\frac{d-2}2} \mathbb E \big[ {\mathscr Z_\8(0)} - { \mathscr Z_T(0)} E_{0} \big( \mathscr Z_\8 \circ \theta_{T,W(T)}  \big) \big]^2  \\ \nn
&=& T^{\frac{d-2}2} \mathbb E \left[ E_0 \big( (\Phi_T(W)-\mathscr Z_T(0)) 
\mathscr Z_\8 \circ \theta_{T,W_T} \right]^2  \\ \nn
&=& T^{\frac{d-2}2} E_0^{\otimes 2} \mathbb E \left[  \big(\Phi_T(W^{\ssup 1})-\mathscr Z_T(0)\big)  \big(\Phi_T(W^{\ssup 2})-\mathscr Z_T(0)\big) 
\mathscr Z_\8 \circ \theta_{T,W_T^{\ssup 1}} \mathscr Z_\8 \circ \theta_{T,W_T^{\ssup 2}} \right] \\ \nn
&=& T^{\frac{d-2}2} E_0^{\otimes 2} \mathbb E \left[  \big(\Phi_T(W^{\ssup 1})-\mathscr Z_T(0)\big)  \big(\Phi_T(W^{\ssup 2})-\mathscr Z_T(0)\big) 
\big( \mathscr Z_\8 \circ \theta_{T,W_T^{\ssup 1}} -1\big) \big(\mathscr Z_\8 \circ \theta_{T,W_T^{\ssup 2}} -1\big) \right]\\ \nn
&\stackrel{}{=}& 
T^{\frac{d-2}2} E_0^{\otimes 4}  \left[  \left( \mathrm e^{\frac{\beta^2}{2} \int_0^T V(W_s^{\ssup 1}-W_s^{\ssup 2})\dd s} -
\mathrm e^{\frac{\beta^2}{2} \int_0^T V(W_s^{\ssup 1}-W_s^{\ssup 4})\dd s} - \mathrm e^{\frac{\beta^2}{2} \int_0^T V(W_s^2-W_s^3)ds} +
\mathrm e^{\frac{\beta^2}{2} \int_0^T V(W_s^{\ssup 3}-W_s^{\ssup 4})\dd s} \right)  \right.  \\ \nn && \qquad \times  f(W_T^{\ssup 1}-W_T^{\ssup 2}) \Big]\quad
{\rm with\; }  f(w)= {\rm Cov}(\mathscr Z_\8(w), \mathscr Z_\8(0))= \frac{\mathfrak C_1}{|w|^{d-2} } \quad {\rm for\ } |w| >1, {\rm \; recall \ } \eqref{eq:covM0}.
\end{eqnarray}
Now, we note that under $P_0^{\otimes 4}$, as $T \to \8$, the couple
$$
\big (W^{\ssup 1}_s-W^{\ssup 2}_s, W^{\ssup 1}_s-W^{\ssup 4}_s, W^{\ssup 2}_s-W^{\ssup 3}_s, W^{\ssup 3}_s-W^{\ssup 4}_s\big )_{s \in [0,T]} \quad {\rm and} \quad \frac{W^{\ssup 1}_T-W^{\ssup 2}_T}{\sqrt T}  
$$
converge in law to the product measure of the law of 
$
\big (W^{\ssup 1}_s-W^{\ssup 2}_s, W^{\ssup 1}_s-W^{\ssup 4}_s, W^{\ssup 2}_s-W^{\ssup 3}_s, W^{\ssup 3}_s-W^{\ssup 4}_s\big )_{s \geq 0}
$
and the Gaussian ${\mathcal N}(0, \sqrt 2 I_d)$, which implies that
$$
\lim_{T \to \8} c_T = 2 \mathfrak C_2 E_0^{\otimes 3} \big( \mathrm e^{\frac{\beta^2}{2} \int_0^\8 V(W_s^{\ssup 1}-W_s^{\ssup 2})\dd s} -
\mathrm e^{\frac{\beta^2}{2} \int_0^\8 V(W_s^{\ssup 1}-W_s^{\ssup 2})\dd s} \big) =0. 
$$
with $\mathfrak C_2$ defined underneath \eqref{eq:asdifnorm}. 
This proves the desired joint convergence of $\int_{\rd} \varphi(x)   \mathscr H_\e(t,x) \;\dd x$ and $\int_{\rd} \varphi(x)   \overline{\mathscr H}_\e(t,x) \;\dd x$ 
 with independent limits in \eqref{eq2:stationary}. \qed

\section{Proof of Proposition \ref{prop:tightness}}\label{sec-tightness}
It is enough to show that  $\{\mathscr H_\e(t,x)\}_{\e>0, x\in \rd}$ forms a tight family, since the convergence part of the proposition will follow from tightness and uniqueness of the limit established in Theorem \ref{th:CVagainstTestFun}. To prove tightness, we appeal to the following tightness criterion which was recently established in \cite{FM17} and was shown to hold in a function space $\mathcal C^\alpha_{\mathrm{loc}}(\rd)$ of distributions with ``local $\alpha$-H\"older regularity" for $\alpha\in \R$. Loosely speaking, this means 
that a  distribution $f\in \mathcal C^\alpha_{\mathrm{loc}}(\rd)$ if and only if 
for any smooth test function $\varphi\in \mathcal C^\infty_c(\rd)$ with compact support and $x\in \rd$, $\lambda^{-d}\langle f,\varphi(\lambda^{-1}(\cdot- x))\rangle \leq C \lambda^{\alpha}$ for $\lambda\sim 0$. In \cite[Theorem 1.1]{FM17})
it was shown there that there is a finite family of smooth and compactly supported functions $\phi$ and $(\psi^{\ssup i})_{1\leq i \leq 2^d}$ so that the following condition holds: 
\begin{theorem}
Let $(f_\e)_\e$ be a family of random linear forms on $C^r_c(\rd)$, let $p\in [1,\infty)$ and $\beta\in \R$ such that $|\beta| < r$. Suppose there is an absolute constant $C<\infty$ such that the following two conditions hold:
\begin{equation}\label{cond1}
\sup_{\heap{x\in \rd}{\e>0}} \E\big[\big|\big\langle f_\e, \phi(\cdot- x)\big\rangle\big|^p\big]^{1/p} \leq C,
\end{equation}
and for all $i=1,\dots,2^d-1$ and $n\in \N$, 
\begin{equation}\label{cond2}
\sup_{\heap{x\in \rd}{\e>0}} \E\bigg[\bigg|\bigg\langle f_\e, \psi^{\ssup i}\bigg(\frac{\cdot- x}{2^{-n}}\bigg)\bigg\rangle\bigg|^p\bigg]^{1/p} \leq C 2^{-nd} \, 2^{-n\beta}.
\end{equation}
Then the family $(f_\e)_\e$ is tight in $\mathcal C^\alpha_{\mathrm{loc}}(\rd)$ for every $\alpha < \beta- \frac d p$.
\end{theorem} 

\noindent{\bf{Concluding the proof of Proposition \ref{prop:tightness}:}} 
We will check the requisite conditions \eqref{cond1} and \eqref{cond2} for $f_\e(y)= \mathscr H_\e(t,y)= \e^{1-\frac d 2} [h_\e(t,y)- \mathfrak h(\xi^{\ssup{\e,t,y}})]$ and $\beta=0$ and $p\in (1,2)$. 
 First remark that in this context $f_\e(y)$  is stationary in the $y$-variable. 
 Next to check the above two requirements, it suffices to show that for any smooth function with compact support (say, in a ball of radius $1$), 
\begin{equation}\label{cond3}
\sup_{\heap{x\in \rd}{\e>0}} \bigg(\E\bigg[\bigg|\int_{\rd} f_\e(y) \psi\big(\frac{y-x}{2^{-n}}\big) \dd y \bigg|^p\bigg]\bigg)^{1/p} \leq C 2^{-nd}.
\end{equation}
Suppose that $f_\e$ is stationary for all $\e>0$. Let $q>1$ verify $p^{-1} + q^{-1}=1$. By H\"older's inequality, the LHS in the above display is bounded above by %(recall that $\psi$ is a continuous 
\[
\begin{aligned}
&\sup_{\heap{x\in \rd}{\e>0}} \E\bigg[\int_{B(x,2^{-n})} |f_\e(y)|^p \dd y \bigg]^{1/p} \left(\int_{\rd} \left|\psi\big(\frac{y-x}{2^{-n}}\big)\right|^q \dd y \right)^{1/q} \\
& \leq \sup_{x\in \rd}  |B(x,2^{-n})|^{1/p}\, \sup_{\e>0}\E\left[|f_\e(0)|^p\right]^{1/p} 2^{-nd/q} \|\psi\|_q
 \leq \|\psi\|_q\, \sup_{\e>0} \E\left[|f_\e(0)|^p\right]^{1/p} 2^{-nd},
\end{aligned}
\]
where we have used stationarity of $f_\e$ in second line. Hence, we observe that tightness of $(f_\e)_\e$ would follow from $L^p(\mathbb P)$-boundedness of $(f_\e)_\e$ for $p\in (1,2)$. 
To this end, we appeal to the $L^p(\mathbb P)$-boundedness of $T^{(d-2)/4}[\log\sZ_T-\log\sZ_\infty]$ for all $p\in (1,2)$ (recall Proposition \ref{prop:LpBoundLogZ}), which, together with 
the identity \eqref{eq:uZ} in turn implies that $f_\e$ enjoys the same property, implying tightness of $(\mathscr H_\e(t,x)_{x\in \mathbb{R}^d})_{\e>0}$ in the space $\mathcal C^\alpha_{\mathrm{loc}}(\rd)$ for all 
$\alpha < -d /2$. Thus, Proposition \ref{prop:tightness} is proved. \qed

\noindent{\bf Acknowledgements:}  The authors would like to thank Ofer Zeitouni (Rehovot/ New York) for very useful feedback and discussions. Research of the third author is funded by the Deutsche Forschungsgemeinschaft (DFG) under Germany's Excellence Strategy {\it EXC 2044-390685587, Mathematics M\"unster: Dynamics-Geometry-Structure}. The authors were partly supported by the French Agence Nationale de la Recherche under grant ANR-17-CE40-0032. The authors acknowledge  the hospitality of ICTS-TIFR Bengaluru during the program Large deviation theory in statistical physics (ICTS/Prog-ldt/2017/8), where the present work was initiated. The second and the third author would like to thank the hospitality of NYU Shanghai where part of the present work was completed during the first author's long term stay during the academic year 2018-19.

%\noindent{\bf{Acknowledgement:}} The authors would like to thank the ICTS, Bangalore for the hospitality during the program {\it{Large deviation theory in statistical physics}}
%(ICTS/Prog-ldt/2017/8), 
%where the present work was initiated. 

\end{document}